\RequirePackage{hyperref}
\documentclass[11pt, reqno]{amsart}
\usepackage{eqnarray,amsmath,mathtools}
\usepackage{amsthm}

\usepackage{amsfonts, amssymb, stmaryrd}
\usepackage{amsbsy}
\hypersetup{colorlinks,linkcolor={blue!50!black},citecolor={red!50!black},urlcolor={blue!80!black}}
\pdfoutput=1

\usepackage[all]{xy}

\DeclareMathAlphabet{\mathpzc}{OT1}{pzc}{m}{it}

\usepackage[paperwidth=8.5in,paperheight=11in,text={6.in,8.5in}
,centering]{geometry}
\usepackage{bm}

\usepackage{color} 

\newtheorem{lemma}[equation]{Lemma}
\newtheorem{theorem}[equation]{Theorem}

\newtheorem{proposition}[equation]{Proposition}
\newtheorem*{corollary*}{Corollary}

\theoremstyle{definition}
\newtheorem{definition}[equation]{Definition}

\newtheorem{remark}[equation]{Remark}
\newtheorem{example}[equation]{Example}
\newtheorem{notation}[equation]{Notation}

\numberwithin{equation}{section}

\newcommand{\fm}{\mathfrak m}

\newcommand{\frs}{\mathfrak s}
\newcommand{\frt}{\mathfrak t}
\newcommand{\fru}{\mathfrak u}
\newcommand{\cP}{{\mathcal P}}
\newcommand{\cO}{{\mathcal O}}
\newcommand{\cI}{{\mathcal I}}

\newcommand{\mcG}{\mathcal{G}}

\newcommand{\N}{\mathbb{N}}
\newcommand{\R}{\mathbb{R}}

\newcommand{\D}{\mathbb{D}}

\newcommand{\PP}{\mathbb{P}}

\newcommand{\bfW}{\mathbf{W}}
\usepackage{mathrsfs} 

\newcommand{\scH}{\mathscr{H}}
\newcommand{\scX}{\mathscr{X}}

\newcommand{\scC}{\mathscr{C}}
\newcommand{\scD}{\mathscr{D}}

\newcommand{\Control}{\mathsf{Crl}}
\newcommand{\Crl}{\mathsf{Crl}}
\newcommand{\A}{\mathsf{A}}

\newcommand{\Vect}{\mathsf{Vect}}
\newcommand{\RelVect}{\mathsf{RelVect}}
\newcommand{\LinRel}{\mathsf{LinRel}}
\newcommand{\Set} {\mathsf{Set}}

\newcommand{\FinSet}{\mathsf{FinSet}}
\newcommand{\sfC} {\mathsf{C}}
\newcommand{\Man} {\mathsf{Man}}

\newcommand{\Graph}{\mathsf{graph}}

\newcommand{\OS}{\mathsf{OS}}
\newcommand{\SSub}{\mathsf{SSub}}

\newcommand{\inp}{\mathrm{inp}}

\newcommand{\out}{\mathrm{out}}

\newcommand{\st}{\mathsf{st}}
\newcommand{\tot}{\mathsf{tot}}
\newcommand{\op}{\mathsf{op}}

\newcommand{\SSubi}{\mathsf{SSub}^\mathsf{int}}
\newcommand{\id}{\mathsf{id}}

\newcommand{\inv}{^{-1}}
\newcommand{\toto}{\rightrightarrows}
\newcommand{\grph}{\mathop \mathrm{graph}}

\DeclareMathOperator{\Hom}{\mathsf{Hom}}

\usepackage{tikz}
\usetikzlibrary{arrows,cd,decorations.markings,
  calc,chains,matrix,positioning,scopes,snakes}
\tikzset{
   tick/.style={postaction={
      decorate,
      decoration={markings, mark=at position 0.5 with {\draw[-] (0,.4ex) -- (0,-.4ex);}}}
   }
}
\newcommand{\tickar}{\begin{tikzcd}[baseline=-0.5ex,cramped,sep=small,ampersand replacement=\&]{}\ar[r,tick]\&{}\end{tikzcd}}
\newcommand{\ltickar}{\begin{tikzcd}[baseline=-0.5ex,cramped,sep=small,ampersand replacement=\&]{}\&\ar[l,tick]{}\end{tikzcd}}

\begin{document}

\title{Networks   of open 
  systems}

\author{Eugene Lerman}
\address{Department of Mathematics, University of Illinois at Urbana-Champaign,
1409 W. Green Street
Urbana, IL 61801}

\begin{abstract}
  Many systems of interest in science and engineering are made up of
  interacting subsystems.  These subsystems, in turn, could be made up
  of collections of smaller interacting subsystems and so on.  In a
  series of papers David Spivak with collaborators formalized these
  kinds of structures (systems of systems) as algebras over
  presentable colored operads \cite{Spivak}, \cite{RupSp}, \cite{VSL}.

  It is also very useful to consider maps between dynamical systems. 
  This is the point of view taken by DeVille and Lerman in the study
  of dynamics on networks \cite{DL1}, \cite{DL2}, \cite{DL3}.  The
  work of DeVille and Lerman was inspired by the coupled cell networks
  of Golubitsky, Stewart and their collaborators
  \cite{Stewart.Golubitsky.Pivato.03, Golubitsky.Stewart.Torok.05,
    Golubitsky.Stewart.06}.

  The goal of this paper is to describe an algebraic structure that
  encompasses both approaches to systems of systems.  More
  specifically we define a double category of open systems and
  construct a functor from this double category to the double category
  of vector spaces, linear maps and linear relations.  This allows us,
  on one hand, to build new  open systems out of collections of
  smaller open subsystems and on the  other to  keep track of maps between
  open systems.  Consequently we obtain synchrony results for open
  systems which generalize the synchrony results of Golubitsky,
  Stewart and their collaborators for groupoid invariant vector fields
  on coupled cell networks.
\end{abstract}

\maketitle 

\tableofcontents

\section{Introduction}

Many systems of interest in science and engineering are made up of
interacting subsystems.  These subsystems, in turn, could be made up
of collections of smaller interacting subsystems and so on.  These
kind of structures have been formalized by David Spivak with
collaborators as algebras over presentable colored operads
\cite{Spivak, RupSp, VSL}. There are several variants of these
operads; they depend on the kinds of systems one is interested in.
Since the subsystems are supposed to receive input from other
subsystems they are conveniently modeled as {\em open} (a.k.a.\ {\em
  control}) systems; we review open systems in Section~\ref{sec:open}.
Informally an open system is a dynamical system that receive inputs
from other systems.  There are several formal models of open systems
starting with collections of vector fields that depend on parameters.
In \cite{VSL} the input-state-output model is used.

One of the fundamental problems in the theory of (closed) dynamical
systems is finding or, failing that, proving the existence of
equilibria, periodic orbits and, more generally, other invariant
submanifolds. This amounts to finding/proving the existence of maps
between dynamical systems. For example, a map from a point to our
favorite closed system is an equilibrium, maps from circles are
periodic orbits, and so on. Thus it is highly desirable to have a
systematic way of constructing maps between dynamical systems. 

One can view a {\em network} as a pattern of interconnection of open
system.  In \cite{VSL} a network is formalized as a {\em morphism} in
the colored operad of wiring diagram --- the morphism encodes the
pattern.  In the work of DeVille and Lerman \cite{DL1,DL2, DL3}, which
was inspired by the coupled cell networks of Golubitsky, Stewart and
their collaborators \cite{Stewart.Golubitsky.Pivato.03,
  Golubitsky.Stewart.Torok.05, Golubitsky.Stewart.06}, networks are
encoded by directed graphs.  In contrast to \cite{VSL} networks in
\cite{DL1,DL2, DL3} are viewed as {\em objects} in a category, and the
main result  is a good notions of a {\em map}
between networks.   The notion leads to a combinatorial recipe for a
construction of maps of closed dynamical systems out of appropriate
maps of graphs.  We will show in this paper that the networks of \cite{DL1, DL2, DL3} can be viewed as
particular kinds of morphisms in a colored operad
(Proposition~\ref{prop:5.5}).  The morphisms of networks of \cite{DL1,
  DL2, DL3}, on the other hand, have no obvious interpretation in the
operadic language.

In this paper we generalize both approaches (directed graphs and
operadic) to networks of open systems.  This allows us, on one hand,
to build new large open systems out of collections of smaller open
subsystems and on the other hand keep track of maps between open systems.
Consequently we obtain synchrony results for open systems which
generalize the synchrony results of Golubitsky, Stewart and their
collaborators for groupoid invariant vector fields on coupled cell
networks (see for example \cite{Stewart.Golubitsky.Pivato.03,
  Golubitsky.Stewart.Torok.05, Golubitsky.Stewart.06}).

Networks of open systems as such are not new. For example, networks of
open systems are implicit in the work of Field \cite{Field}.  They are
also implicit in the work of Golubitsky, Pivato, Torok and Stewart
\cite{Stewart.Golubitsky.Pivato.03, Golubitsky.Stewart.Torok.05} and
their collaborators.  Special cases of networks of open systems
present in the coupled cell network formalism were made explicit in
\cite{DL1, DL2, DL3}.
Maps between open systems are not new either.  For example the
category of open systems has been explicitly introduced by Tabuada
and Pappas \cite{TP}.  

What is new in this paper is a general notion of maps between networks
(Definition~\ref{def:map_of_networks}) and a dynamical/control system
interpretation of these maps (Theorem~\ref{thm:main}).  We frame this
notion in terms of double categories.  In particular the results of
this paper subsume and extend the results of \cite{DL1, DL2, DL3}, as
we explain in
Section~\ref{sec:networks}.\\[-2pt]

The paper assumes that the reader is comfortable with viewing
continuous time dynamical systems as vector fields on manifolds.  By
necessity the paper also uses a certain amount of category theory,
which we try to keep down to a minimum.  We expect the reader to be
comfortable with the universal properties of products and coproducts
and have a nodding acquaintance with 2 categories, but not much more
than that.  Some of the results of the paper are expressed in the
language of symmetric monoidal categories and the corresponding
colored operads.  A reader who may be
unfamiliar with monoidal categories may safely skip the corresponding
parts of the paper.  We also use the language of (strict) double
categories.  Since double categories are somewhat less common, we do
not expect any familiarity with them on the part of the reader.
Strict double categories are reviewed in Section~\ref{sec:double}.

\subsection*{Organization of the paper}
We start by recalling the definition of an open system
(Definition~\ref{def:open_sys}) and reviewing the category of open
systems of Tabuada and Pappas \cite{TP}.  We then constructing a
symmetric monoidal category $(\SSubi)^\op$ whose objects are
surjective submersions.  In a coordinate-free approach to control
theory a surjective submersion $a$ gives rise to a vector space of
$\Crl(a)$ of control (i.e., open) systems.  We extend the assignment
\[
a \mapsto \Crl(a)
\]
to a morphism of symmetric monoidal categories
\[
\Crl: (\SSubi)^\op \to \Vect,
\]
where $\Vect$ is the 
category of real vector spaces and linear maps with the monoidal
product being given by direct sum $\oplus$. 

Recall that a symmetric monoidal category $\A$ defines a colored
operad $\cO(\A)$.  We interpret a morphism in the operad $\cO
((\SSubi)^\op)$ as a {\em pattern of interconnection of open systems}
and think of it as a {\em network} of open systems.  The monoidal
functor $\Crl$ turns the colored operad $\cO(\Vect, \oplus)$ into an
algebra over the operad $\cO((\SSubi)^\op)$
(Section~\ref{sec:operad}).  

In Section~\ref{sec:lists} we review the category of lists
$\FinSet/\scC$ in a category $\scC$.  The objects of $\FinSet/\scC$
are finite unordered lists of objects of $\scC$.  This is done to
facilitate the comparison of the operad $\cO ((\SSubi)^\op)$ with the
operad of wiring diagrams of \cite{VSL}.  There are also other reasons
for introducing the categories of lists that will become apparent later.
We then revisit the algebra $\cO\Crl: \cO((\SSubi)^\op)\to \cO\Vect$
introduced earlier in Subsection~\ref{subsec:5.2}.

We carry out the comparison of the operad
$\cO((\SSubi)^\op)$ with the operad of wiring diagrams in
Section~\ref{sec:WD}.  The main difference between the two operads and
their respective algebras is philosophical. Namely, the approach of
\cite{VSL} is to treat an open system as a black box --- the space of
internal states is completely unknown while the algebra supplies all
possible choices of internal state spaces.  By contrast in this paper
we treat the space of internal states (and the total space) as known
and have the algebra supply the possible choices of dynamics on a
given total space.

The next two sections are technical. The main result of
Section~\ref{sec:7} is Lemma~\ref{thm:5.9}.  This lemma, in effect, is
half of the proof of the main theorem of the paper,
Theorem~\ref{thm:main}.  The results of Section~\ref{sec:7} are used
to motivate the introduction of double categories, which is carried
out in Section~\ref{sec:double}.  The two main results of
Section~\ref{sec:double} are Lemma~\ref{thm:5.9mrk2} (which is a
reformulation of Lemma~\ref{thm:5.9} in terms of double categories)
and Lemma~\ref{lemma:crl_ssub_relvect}.

Finally in
Section~\ref{sec:net_open} we introduce our notion of maps between
networks (Definition~\ref{def:map_of_networks}) and interpret it in
terms of maps of open systems (Theorem~\ref{thm:main}).  In
Section~\ref{sec:networks} we show that the networks of \cite{DL1}
(hence the coupled cell networks of Golubitsky, Stewart et al.) are a
special case of the networks in the sense of
Definition~\ref{def:network}. We then show that Theorem~3 of
\cite{DL1} (which is the main result of that paper) is an easy
consequence of Theorem~\ref{thm:main}.

\subsection*{Acknowledgments} I thank Tobias Fritz, Joachim Kock and
John Baez for a number of helpful comments.   

This paper started out as a joint project with David Spivak.  An
earlier version of the paper is \cite{LS}.

\section{Open systems} \label{sec:open}
In this section we  define open/control systems and maps
between them. We then construct the functor $\Crl$ which assigns to a
surjective submersion the vector space of all control systems
supported by the submersion.  The tricky part is figuring out the
target category of $\Crl$.\\

Informally an {\em open} (a.k.a. a {\em control}) system is a
dynamical system that receive inputs from other systems.  There are
several formal models of open systems.  The simplest has the following
form.  Fix a manifold $M$ of internal states of the system and another
manifold $U$ (the space of parameters).  An open system is a map (of
an appropriate regularity)
\[
R: U\to \scX (M),
\]
where $\scX(M)$ is a space of vector fields on $M$.  Any map
$R:U\to \scX (M)$ corresponds to a map 
\[
\hat{R}:U\times M\to TM
\]
with the property that 
\begin{equation}\label{eq:2.1}
\hat{R}(u,m)\in T_mM\quad\textrm{ for all } (u,m)\in U\times M.
\end{equation}
Equation \eqref{eq:2.1} is equivalent to the commutativity of the  diagram
\begin{equation}\label{eq:2.2}
\xy 
(-11, 6)*+{M\times U}="1"; 
(6, 6)*+{TM} ="2"; 
(6,-5)*+{M}="3"; 
{\ar@{->}_{ p} "1";"3"}; 
{\ar@{->}^{\,\,\hat{R}} "1";"2"}; 
{\ar@{->}^{\pi} "2";"3"}; 
\endxy 
\end{equation}
where $p:M\times U\to M$ is the projection on the first factor and
$\pi:TM\to M$ is the canonical projection from the tangent bundle of
$M$ to its base.  We think of the manifold $M\times U$ as the total
space of the open system with the factor $U$ representing the space of
inputs or of control variables (we use the words ``inputs'' and ``controls'' interchangeably).  However in many control systems of interests the
factorization of the total space into internal states and inputs is
not natural.  For this reason we adopt a somewhat more general
definition of a open/control system.  Note that the map $p$ in
\eqref{eq:2.2} is a surjective submersion.  Following Brockett
\cite{Brockett} and Tabuada and Pappas \cite{TP} we define a
(continuous-time) open system as follows:
\begin{definition}[open system] \label{def:open_sys} A continuous time
  {\sf open system} on a surjective submersion $p: Q\to M$ is a smooth
  map $ F: Q \to TM$ so that $ F(q) \in T_{p(q)} M$ for all $q\in Q
  $. That is, the following diagram commutes:
\begin{equation}\label{eq:3.2}
 \xy 
(-10, 6)*+{Q}="1"; 
(6, 6)*+{TM} ="2"; 
(6,-5)*+{M}="3"; 
{\ar@{->}_{ p} "1";"3"}; 
{\ar@{->}^{F} "1";"2"}; 
{\ar@{->}^{\pi} "2";"3"}; 
\endxy \qquad.
\end{equation}
Thus an {\sf open system} (or a {\sf control system}) is a pair
$(Q\xrightarrow{p}M, F)$ where $p$ is a surjective submersion and
$F: Q\to TM$ is a smooth map satisfying \eqref{eq:3.2}.  We refer
to the manifold $Q$ as the {\sf total space} and of the
manifold $M$ as the {\sf state space.  }
\end{definition}

\begin{remark}\label{rem:fixed_input_space}
  In the case when  $Q= M\times U$ for some manifold $U$  and the surjective submersion $p: M\times U \to M$
  is the projection on the first factor, we  think of $U$ as the
  space of input variables and say that the open system $F:M\times
  U\to TM$ is an {\sf open system with a choice of factorization of
    the total space into inputs and states}.

  Note that in general even if a surjective submersion $p:Q\to M$ is a
  trivial fiber bundle with a typical fiber $U$ there may not be a preferred
  choice of a factorization $Q\xrightarrow{\simeq} M\times U$.  The
  lack of natural factorization of variables of open systems into
  inputs and states 
  has been emphasized by Willems
  \cite{Willems}.
\end{remark}

\begin{remark}
Fix a surjective submersion $p:Q\to M$.   The set
\begin{equation}
\Crl(Q\xrightarrow{p} M):= \{F: Q\to TM \mid \pi\circ F = p\},
\end{equation}
of all control systems for the given submersion $p$ has the structure
of an infinite dimensional real vector space.
\end{remark}
\begin{remark}
  If the surjective submersion $Q\to M$ is the identity map $\id: M\to
  M$ then the space of open systems $\Control(M\xrightarrow{\id}M)$ is
  the space $\scX(M)$ of vector fields on the manifold $M$. Thus
  closed systems (i.e., vector fields) can be thought of as open
  systems whose space of inputs is a point (provided we suppress the
  diffeomorphism $M\times \{*\} \xrightarrow{\,\simeq \,} M$.)
\end{remark}

Just as vector fields have trajectories, so do open systems.
\begin{definition} A {\sf trajectory} of an open system
  $(Q\xrightarrow{p} M, F:Q\to TM)$ is a curve $\gamma:(a,b)\to Q$ so that 
\[
 \frac{d}{dt} (p\circ \gamma)\, (t) = F(\gamma (t)).
\]
for all $t$ in the open interval $(a,b)$.
\end{definition}

\begin{remark}\label{rmrk:2.10} \mbox{}\\
\begin{itemize}
\item If $Q= M$ and $p$ is the identity map, the definition above
  reduces to the definition of an integral curve of a vector field.
\item if $Q= M\times U$ and $p:M\times U\to M$ is the projection on
  the first factor then a trajectory $\gamma$ of an open system $F:M\times U\to
  TM$ is of the  form 
\[
\gamma(t) = (x(t), u(t))
\] 
with
\[
 \frac{d}{dt} x(t)  = F(x(t), u(t)).
\]
Such a definition of a trajectory is very common in the control theory
literature.
\end{itemize}
\end{remark}

Recall that a vector field $X$ on a manifold $M$ is {\sf $f$-related}  to a vector field $Y$ on a manifold $M$ (where $f:M\to N$ is a smooth map of manifolds) if \begin{equation}\label{eqn:intertwine}
Tf\circ X = Y\circ f.
\end{equation}
Here and elsewhere in the paper $Tf: TM\to TN$ denotes the
differential of $f$.  Equivalently the diagram
\begin{equation}\label{eqn:CDs}
\xy 
(-10, 6)*+{M}="1"; 
(6, 6)*+{TM} ="2"; 
(-10, -6)*+{N}="3"; 
(6,-6)*+{TN.}="4"; 
 {\ar@{->}_{f} "1";"3"}; 
{\ar@{->}^{X} "1";"2"}; 
{\ar@{->}^{Tf} "2";"4"}; 
 {\ar@{->}_{Y} "3";"4"}; 
\endxy
\end{equation} 
commutes.  Recall also that if $X\in \scX(M)$ is
$M\xrightarrow{f}N$-related to $Y\in \scX(N)$ and $\gamma$ is a
trajectory (i.e., an integral curve) of $X$ then $f\circ \gamma$ is an
integral curve of $Y$.   

The analogues results holds for open systems.  To state it we need to
first recall the notion of a map of submersions and then the notion of
morphism of open systems, where we follow \cite{TP}.
\begin{definition}\label{def:mor_ssub}
 A {\sf morphism} $f$
  from a submersion $p:Q\to M$ to a submersion $p': Q'\to
  M'$ is a pair of maps $f_{\tot}: Q\to Q'$, $f_{\st}: M\to
  M'$ so that the following square commutes:
\[
\begin{tikzcd}
Q\ar[r,"f_\tot"]\ar[d,"p"']&Q'\ar[d,"p'"]\\
M\ar[r,"f_\st"']&M'
\end{tikzcd}
\]
\end{definition}

\begin{definition}[The category $\SSub$ of surjective submersions]
  Definition~\ref{def:mor_ssub} allows us to turns the collection of
  surjective submersions into a category. We denote it by $\SSub$.
  Explicitly an object $a$ of the category $\SSub$ is surjective
  submersion 
\[
a= (a_\tot \xrightarrow{p_a} a_\st) 
\] 
(Here and elsewhere in the paper $a_\tot$ stands for the total space
of $a$, $a_\st$ for the state space and $p_a$ is the submersions).  A
morphism $a\xrightarrow{f} b$ in the category $\SSub$ is a map of
submersions. That is, it is a pair of smooth maps $(f_\tot, f_\st)$ so
that $p_b\circ f_\tot = f_\st \circ p_a$.
\end{definition}
\begin{notation}\label{note:2.15}
  We will use two types of notation for surjective submersions: $a=
  (a_\tot \xrightarrow{p_a} a_\st)$ and $Q\xrightarrow{p} M$. The usage
  will depend on convenience.
\end{notation}

\begin{remark}
  Surjective submersions are also known as {\sf fibered manifolds}.
  The term ``fibered manifold'' goes back to Seifert and Whitney and
  has been in use  since the early 1930's.
\end{remark}

\begin{definition}[cf. \protect{\cite[Definition~4.1]{TP}}]
\label{def:dyn_morph} 
A {\sf morphism}
  from an open system $(Q\to M, F:Q\to TM)$ to an open system $(Q'\to M', F':Q'\to TM')$
  is a morphism of submersions $f= (f_{\tot}, f_{\st}):(Q\to M)
  \to (Q'\to M')$ for which the following diagram commutes:
\[
\begin{tikzcd}
Q\ar[r,"f_\tot"]\ar[d,"F"']&Q'\ar[d,"F'"]\\
TM\ar[r,"Tf_\st"']&TM'
\end{tikzcd}
\]
In this case, we say that the open systems $(Q\to M,F)$ and $(Q'\to
M',F')$ are {\em $f$-related}.
\end{definition}

\begin{remark}
It is easy to see that if $f:(Q\to M, F) \to (Q'\to M')$ is a 
morphism of an open systems and $\gamma$ is a trajectory of $F$ then
$f\circ \gamma$ is a trajectory of $F'$.
\end{remark}

We can now in position to recall the definition of the category $\OS$
of open systems (it is called $\mathsf{Con}$ for {\em control} in \cite{TP}).

\begin{definition}[The category $\OS$] The objects of the
  category $\OS$ of open systems are
  open systems $(Q\to M,F)$ as in Definition~\ref{def:open_sys}.
  Morphisms of $\OS$ are morphisms of open systems as in 
  Definition~\ref{def:dyn_morph}.
\end{definition}

\begin{remark}
The categories $\SSub$ of surjective submersions and the category
$\OS$ of open systems have finite products.  The product of two
surjective submersions $Q_i\xrightarrow{p_i} M_i$, $i=1,2$, is the
submersion
\[
p_1\times p_2:Q_1\times Q_2\to M_1\times M_2, \qquad
(p_1\times p_2)\, (q_1, q_2) = (p_1(q_1), p_2(q_2)).
\]
The product of two open systems $(Q_i\xrightarrow{p_i} M_i, F_i:Q_i
\to TM_i)$, $i=1,2$, is the open system
\[
(Q_1\times Q_2\xrightarrow{p_1\times p_2} M_1\times M_2,\quad F_1
\times F_2:Q_1 \times Q_2 \to TM_1\times TM_2 \simeq T(M_1\times
M_2)),
\]
where
\[
 (F_1 \times F_2)\, (q_1, q_2) = (F_1(q_1), F_2(q_2))
\]
for all $(q_1, q_2)\in Q_1\times Q_2$. 
\end{remark}

\begin{remark}
  We have the evident forgetful functor $u: \OS \to \SSub$ from open
  systems to submersions that forgets the dynamics.  That is, on
  objects the functor $u$ is given by
\[
u(a, F) = a.
\]
Note that the functor $u$ preserves finite products. 
\end{remark}

For a surjective submersion $a$ the fiber $u\inv
(a)$ of the functor $u: \OS \to \SSub$ is (isomorphic to) the space
$\Crl(a)$.  This suggest that the assignment which
sends a surjective submersion $a$ to the space of open systems $\Crl(a)$
should extend to a functor from the category $\SSub$ to some category.
The objects of the target category should be real vector spaces since
$\Crl(a)$ is a real vector space.  But what are the morphisms?  A map
$f:a\to b$ between two submersions does not in
general give rise to a linear map from $\Crl(a)$ to $\Crl(b)$.  
However a morphism of
submersions $f:a\to b$ defines a linear
relation
\begin{equation}
\Crl(f):=\{(G, F)\in \Crl(b)\times \Crl(a) \mid 
F, G \textrm{ are }f\textrm{-related}\}.
\end{equation}
This suggests that the assignment $a\mapsto \Crl(a)$ extends to a
functor from the category $\SSub$ of submersions to the category whose
objects are vector spaces and morphisms are linear relations.  This is
not quite correct.  If $f:a\to b$ and $g:b\to c$ are two morphisms
of submersions then it easy to see that
\[
\Crl(g)\circ \Crl(f) \subset \Crl(g\circ f).
\]
However in general there is no reason for the inclusion to be an
equality of linear subspaces. In fact  the inclusion can be strict (see
Example~\ref{ex:2.25} below).
Thus our best hope is to make $\Crl$ into a lax 2-functor with the
target 2-category
$\mathsf{LinRel}$ of vector spaces, linear relations and inclusions.  
We now proceed to formally define the 2-category $\LinRel$.

\begin{definition}[The 2-category $\LinRel$ of real vector spaces,
  linear relations and inclusions]\label{def:lin-rel}
  The objects of the category $\LinRel$ are (real) vector spaces. A
  1-morphism from a vector space $V$ to a vector space $W$ is a subspace
  $R\subset W\times V$. We write $R:V\tickar W$ and say that $R$ is a
 {\sf linear relation from $V$ to $W$}. 

Given two linear relations $S\subset Z\times Y$ and
  $R\subset  Y\times X$ we define their composite $S\circ R$ to be the
  linear relation
\[
S\circ R :=\{(z,x) \in Z\times X\mid \textrm{ there exist } (z,y) \in S 
\textrm{ and }
(y',x) \in R\textrm{ with } y=y'\}.
\]
It is easy to see that the composition of relations is associative;
hence vector spaces and linear relations form a category.  

A 2-morphism in $\LinRel$ from a relation $R:V\tickar W$ to a relation
$S: V\tickar W$ is a linear inclusion $R\hookrightarrow S$.  We define
the vertical composition of an inclusion $R_1\subset R_2 \subset
Y\times X$ followed by the inclusion $R_2 \subset R_3 \subset Y\times
X$ to be the inclusion $R_1\subset R_3$.  It is easy to check that if
$S\subset S'\subset Z\times Y$ and $R\subset R' \subset Y\times X$ are
two pairs of inclusions of linear relations then
\[
S\circ R\subset S'\circ R'.
\]
Consequently vector spaces, linear relations and inclusions
form  a (strict) 2-category.   We denote it by $\LinRel$.
\end{definition}

\begin{remark}\label{rmrk:2.24}
  There is a functor $\Graph$ from the category $\Vect$ of
  vector spaces and linear maps to the underlying 1-category of
  $\LinRel$.  The functor does nothing on objects and is defined on
  arrows by
\[
\Graph(W\xleftarrow{T}V):= \{(w,v)\in W\times V \mid w= T(v)\}.
\]
\mbox{} \hfill  $\Box$
\end{remark}

Given two composible morphisms of submersions $f:a\to b$ and $g:b\to
c$ it is easy to see that
\[
\Crl(g)\circ \Crl(f) \subset \Crl(g\circ f).
\]
In general there is no reason for the inclusion to be an equality of
linear subspaces.  Hence $\Crl$ is a lax 2-functor.  Here is an
example.

\begin{example} \label{ex:2.25}
Recall that an open system on a submersion of the form
  $id_M:M\to M$ is a vector field on the manifold $M$.  Now consider
  the pair of embeddings
\[
f:  \R \hookrightarrow \R^2, \qquad f(x) = (x,0)
\]
and
\[
g: \R^2\hookrightarrow \R^3, \qquad g(x,y) = (x,y,0).
\]
The constant vector field $\frac{d}{dx}$ on $\R$ is $(g\circ
f)$-related to any vector field $v$ on $\R^3$ with $v(x,0,0) =
\frac{\partial}{\partial x}$. However, such a vector field $v$ need
not be tangent to the $xy$-plane and thus need not be $g$-related to
any vector field on $\R^2$.  Thus in this case
\[
\Crl(g)\circ \Crl(f) \subsetneq \Crl(g\circ f).
\]
\end{example}

\begin{remark} The lax functor $\Crl:\SSub\to\LinRel$ carries all the
  essential information of the forgetful functor $u:\OS\to\SSub$. That is,
  for each object $a\in\SSub$, the fiber
  $u\inv (a)$ is isomorphic to $\Crl(a)$, and to each morphism $f: a\to
  b$ we can associate the linear relation
  $\Crl(f)$. Thus the functor $u$ is in some ways akin to the
  Grothendieck construction of $\Crl$, but $u$ is not a fibration of
  categories.
\end{remark}

\section{Interconnections and networks} 

We take the point of view that a {\em network} is a pattern of
interconnection of a collection of open systems.  The goal of this
section is to make the previous sentence precise. The idea is to start
with a finite unordered list of surjective submersions.  Formally such
a list is a map $\tau:X\to \SSub$, where $X$ is a finite set and
$\SSub$ is the category of surjective submersions defined above. That
is, $\tau(x)$ is a surjective submersion for every $x\in X$. A pattern
of interconnection is then an appropriate map of surjective
submersions $\psi:a\to \prod_{x\in X} \tau(x)$.  To explain what maps
are appropriate and the intuition behind this definition we start with two
examples.

\begin{example}\label{ex:3.1}
  Let $M, U, V$ be manifolds.  Then the projection on the first factor
  \[
M\times U\times V\to M 
\]
is a surjective submersions.  Consider an
  open system $F:M\times U\times V\to TM$.  Let $\phi: M\to V$ be a
  smooth map. Then the map $G: M\times U\to TM$ defined by
\[
G(m,u) = F(m, u,\phi (m))
\]
is an open system on the submersion $M\times U\to M$.

Note also that 
\[
G = F\circ \varphi,
\]
where $\varphi:M\times U\to M\times U\times V$ is given by
\[
\varphi(m,u) = (m, u, \phi(m)).
\]
We therefore view $\varphi$ as defining a pattern of interconnection of
opens systems.  Namely $\varphi$ gives rise to the linear map
\[
\varphi^*: \Control(M\times U\times V \to M)\to \Control(M\times U\to M),
\qquad \varphi^*F:= F\circ \varphi.
\]
\end{example}

\begin{example}\label{ex:3.2}
  Example~\ref{ex:3.1} above gives us a way to view a vector field on a product
  of two manifold as a pair of interconnected open systems, that is,
  as a network.  Namely let $M, N$ be two manifolds and let $X:
  M\times N\to TM\times TN$ be a vector field on their product.  Then
  $X$ is of the form
\[
X(m,n) = (F(m,n), G(n,m))
\]
where $F:M\times N\to TM$ and $G:N\times M\to TN$ are the appropriate
smooth maps.  In fact it is easy to see that $F$ and $G$ are open
systems.  Moreover it should be intuitively clear that $X$ is obtained
from $F$ and $G$ by plugging the states of the open system $F$ into
the inputs of the open system $G$ and the states of $G$ into the
inputs of $F$.  More precisely consider the product open system
\[
F\times G: (M\times N)\times (N\times M) \to TM\times TN, \quad
((m,n),(u,v))\mapsto (F(m,u), G(n,v)) .
\]
Then
\[
X  = (F\times G)\circ \varphi, 
\]
where 
\[
\varphi: M\times N\to (M\times N)\times (N\times M)
\]
is defined by 
\[
 \varphi(m,n) := ((m,n), (n,m)).
\]
\mbox{}\hfill $\Box$
\end{example}

The notion of an interconnection map informally introduced above
easily generalizes to maps between more general submersions.
\begin{definition}\label{def:intercon}
  A morphism $\varphi= (\varphi_\tot, \varphi_\st):
  a \to b$ between two
  submersions is a {\sf interconnection morphism} if $\varphi_\st$ is
  a diffeomorphism.
\end{definition}

\begin{remark}
  We are interested in interconnection morphisms $\varphi$
  with the property that $\varphi_\st$ is the identity map.  We fear,
  however, that {\em requiring} $\varphi_\st$ to be the identity outright
  may cause trouble.
\end{remark}
\begin{remark}
  It may be that Definition~\ref{def:intercon} of an ``interconnection
  map" is a bit too general. For example, one may want to additionally
  insist that $\varphi_\tot:a_\tot\to b_\tot$ is an embedding, though
  we will not need this restriction in what follows.  The additional
  requirement that $\varphi_\tot$ is an embedding would capture the
  idea that after plugging states into inputs the total space of the
  systems (inputs and states) should be smaller.
\end{remark}

\begin{remark} \label{rmrk:3.6} An interconnection morphism
  $\varphi:a\to b$ of Definition~\ref{def:intercon} gives rise to a
  linear map
\[
\varphi^*: \Crl(b) \to \Crl(a).
\]
It is given by
\[
\varphi^*F:= T(\varphi_\st)\inv \circ F \circ \varphi_\tot.
\]
Note that 
\[
\Crl(\varphi) =\{(\varphi^*F, F) \mid 
F\in \Crl(b)\} = \Graph(\varphi^*)
\]
(q.v.\ Remark~\ref{rmrk:2.24}).
\end{remark}

\begin{definition}
[The category $\SSubi$ of submersions and interconnection maps]
  The collection of interconnection maps is closed under compositions.
  Consequently the collections of surjective submersions and
  interconnection maps forms a subcategory of the category $\SSub$.  We
  denote it by $\SSubi$.
\end{definition}

\begin{remark}\label{rmrk:SSub_as_double}
  The subcategory $\SSubi$ of $\SSub$ has the same objects as the
  category $\SSub$.  For any two composible morphisms
  $c\xleftarrow{\psi}b \xleftarrow{\varphi}a$ of $\SSubi$ we have
\[
(\psi\circ \varphi)^* = \varphi^*\circ \psi^*.
\]
Therefore we have another way to extend the assignment
\[
  \SSub \ni a\mapsto \Crl(a)
\]
to a functor.  Namely we have an evident  functor 
\[
(\SSubi)^\op \to \Vect,
\]
which is defined on arrows by 
\begin{equation}\label{eq:3.9}
(b\xleftarrow{\psi}a)\mapsto (\Crl(b)\xrightarrow{\psi^*} \Crl(a))
\end{equation}
(see Remark~\ref{rmrk:3.6}).
At the risk of causing a temporary confusion we will also denote this
functor by $\Crl$.  We will see later in the paper
(Lemma~\ref{lemma:crl_ssub_relvect}) that the functors
\[
\Crl:(\SSubi)^\op \to \Vect\qquad\textrm{and}\qquad
\Crl:\SSub \to \LinRel
\]
are components of a single morphism of double categories.
\end{remark}
We are now in position to define a network of open systems.

\begin{definition} \label{def:network}
  A {\sf network of open systems} is an unordered list of
  submersions $\tau:X\to \SSub$ indexed by a finite set $X$ together
  with an interconnection morphism $\psi:b\to \prod_{x\in X} \tau(x)$.
\end{definition}

\begin{example} \label{ex:3.11} 
Examples~\ref{ex:3.1} and \ref{ex:3.2} are both
  examples of networks of open systems in the sense of
  Definition~\ref{def:network}.

  In Example~\ref{ex:3.1} the ``list'' of submersions consists of the
  single submersion $M\times U\times V \to M$, which we can think of a
  map $\tau:\{*\}\to \SSub$ with $\tau(*) = (M\times U\times V \to
  M)$.  The submersion $b$ is $ (M\times U\to M)$.  The interconnection
  map $\varphi:b\to \tau(*)$ is defined by $\varphi_{st} = id_M$, 
\[
\varphi_{tot}:M\times U\to M\times U\times V,\qquad \varphi_{tot}(m,u)
= (m,u,\phi(m)).
\]
In Example~\ref{ex:3.2} the list of submersions is the map
$\tau:\{1,2\}\to\SSub$ with
\[
\tau(1) = (M\times N\to M),\qquad \tau(2) = (N\times M \to M).
\]
The submersion $b$ is $id:M\times N\to M\times N$.
The interconnection map $\varphi:b\to \tau(1)\times \tau(2)$ is defined by 
$\varphi_{\st} = id_{M\times N}$ and 
\[
\varphi_{\tot}:M\times N\to (M\times N)\times (N\times M),\qquad
\varphi_{\tot}(m,n) := ((m,n), (n,m)).
\]

\end{example}

\begin{remark}[Networks as patterns of
  interconnections]\label{rmrk:net_pattern}
A network is a pattern of interconnection of open systems in the
following sense.  Let $a_1,\ldots, a_n$ be a collection of surjective
submersions (so here $X=\{1, \ldots, n\}$ and $\tau:\{1, \ldots, n\} \to
  \SSub$ is given by $\tau(i) = a_i$). Let $\psi:b\to \prod_{i=1}^n
  a_i$ be an interconnection morphism.  Then given a collection of
  open systems $\{F_i\in \Crl(a_i)\}_{i=1}^n$ we get an open system
  $F\in \Crl(b)$ which is defined by
\begin{equation}\label{eq:3.10}
F = \psi^* (F_1\times \cdots \times F_n).
\end{equation}
We will say more about networks as patterns in the next section.
\end{remark}

We finish the section with one more example of a network.  We'll come
back to this example in Example~\ref{ex:9.1} and Example~\ref{ex:5.9}.

\begin{example}\label{ex:3.13}  
  Let $M, U$ be two smooth manifolds, $p:M\times U\to M$ a trivial
  fiber bundle and $\phi:M\to U$ a smooth map.  Let $X$ be a three
  element set, $X=\{1, 2, 3\}$, and $\tau:X\to \SSub$ the constant map
  with $\tau(1)= \tau(20= \tau(3) = (p:M\times U\to M).$ We choose $b$
  to be the trivial submersion $\id_{M^3}:M^3 \to M^3$ and
\[
\psi: b\to \prod_{x\in X}\tau(x) \simeq ((M\times U)^3 \to M^3)
\]
to be the interconnection morphism with $\psi_\tot$ given by
\[
\psi_\tot(m_1, m_2, m_3)= ((m_1, \phi(m_2)), (m_2, \phi(m_1)),(m_3, \phi(m_2))).
\]
Then $(\tau:X\to \SSub, \psi:b\to \prod_{x\in X}\tau(x))$ is a network
in the sense of Definition~\ref{def:network}.

Note that for any three open systems $F_1, F_2, F_3 \in \Crl(M\times
U\to M)$ the system
\[
F =\psi^* (F_1 \times F_2\times F_3)
\]
is a vector field on $M^3$ given by
\[
F(m_1, m_2, m_3)= (F_1(m_1, \phi(m_2))F_2(m_2, \phi(m_1))F_3(m_3, \phi(m_2))).
\]
\end{example}

\section{Networks as morphisms in a colored operad}
\label{sec:operad}

We now give networks in the sense of Definition~\ref{def:network} an
operadic interpretation.  Recall that if $\sfC$ is a symmetric
monoidal category with the monoidal product $\otimes$ the
corresponding (representable colored) operad $\cO\sfC$ has the same
objects as $\sfC$.  For any objects $a_1,\ldots, a_n, b$ of $\sfC$ the
set $\Hom_{\cO\sfC}(a_1,\ldots,a_n:b)$ of morphisms in the operad
$\cO\sfC$ from $a_1,\ldots, a_n$ to $b$ is defined to be $\Hom_\sfC
(a_1\otimes\cdots \otimes a_n, b)$:
\[
\Hom_{\cO\sfC}(a_1,\ldots,a_n:b):= \Hom_\sfC (a_1\otimes\cdots \otimes a_n, b).
\]
Since $\SSub$ has finite products, it is a Cartesian symmetric monoidal
category.  It is easy to see that the product of two interconnection
morphisms is again an interconnection morphisms.  Consequently
$\SSubi$ inherits from $\SSub$ the structure of a symmetric monoidal
category.\footnote{Note that $\SSubi$ is not Cartesian.  The issue
  is that if $a_1$, $a_2$ are two surjective submersions then the two
  projection maps $\pi_i:a_1\times a_2 \to a_i$ ($i=1,2$) are not (in
  general) morphisms in $\SSubi$, so $\SSubi$ does not have products.}
Consequently the opposite category $(\SSubi)^\op$ is also symmetric
monoidal.  Now for any $n+1$ objects $a_1,\ldots, a_n, b\in \SSubi$ a
morphism in the operad $\cO( (\SSubi)^\op)$ is exactly a network in
the sense of Definition~\ref{def:network}.

We next interpret \eqref{eq:3.10} in terms of an algebra over the
operad $\cO((\SSubi)^\op)$.  In order to carry this out we view the category
$\Vect$ of (real) vector spaces as a symmetric monoidal category with
the monoidal product being the direct sum $\oplus$.%

\begin{lemma}\label{lem:crl_lax}
The functor 
\[
\Crl: ((\SSubi)^\op, \times) \to (\Vect, \oplus) 
\]
defined by \eqref{eq:3.9} is a lax monoidal functor.
\end{lemma}

\begin{proof}
For any two surjective submersions $a,b$ we have a linear map
\[
\Crl_{a,b}: \Crl(a)\oplus \Crl(b) \to \Crl(a\times b)
\]
which is given by
\[
\Crl_{a,b}(F,G) := F\times G
\]
for all $(F,G)\in \Crl(a)\oplus \Crl(b)$.  It is easy to see that for
any two interconnection morphisms $f:a\to a'$, $g:b\to b'$ the diagram
\[
\xy
(-20, 10)*+{\Crl(a')\oplus \Crl(b')}="1"; 
(20,10)*+{\Crl(a'\times b')}="2";
(-20, -10)*+{\Crl(a)\oplus \Crl(b)}="3"; 
(20,-10)*+{\Crl(a\times b)}="4";
{\ar@{->}^{ \Crl_{a',b'}}"1";"2"};
{\ar@{->}_{f^*\oplus g^*} "1";"3"};
{\ar@{->}^{(f\times g)^*} "2";"4"};
{\ar@{->}_{\Crl_{a,b} } "3";"4"};
\endxy
\]
commutes.  
\end{proof}
Since the functor $\Crl: (\SSubi)^\op\to \Vect$ is monoidal, it induces
a map of colored operads
\[
\cO\Crl: \cO((\SSubi)^\op)\to \cO\Vect.
\]
Therefore for any morphism 
\[
\psi\in \Hom_{\cO((\SSubi)^\op)}(a_1,\ldots, a_n;b) 
\]
we get a morphism 
\[
\cO\Crl(\psi) \in
\Hom_{\cO\Vect}(\Crl(a_1),\ldots, \Crl(a_n); \Crl(b))\equiv
\Hom_\Vect(\oplus_i \Crl(a_i), \Crl(b)). 
\]
The linear map $\cO\Crl(\psi):\oplus_i \Crl(a_i)\to \Crl(b)$  is given by
\[
\cO\Crl(\psi)(F_1,\ldots, F_n)= \psi^* (F_1\times\cdots \times F_n)
\]
for any $(F_1,\ldots, F_n)\in \oplus_{i=1}^n \Crl(a_i)$.\\

We would  next like to  give and justify a meaningful notion of a map between
networks.  Our strategy is to first discuss morphisms between lists of
submersions. We carry this out in the next section.  Maps between
networks themselves will be defined in Section~\ref{sec:net_open} after
further preparation.

\section{Categories of lists}\label{sec:lists}

Think of sets as discrete categories. Then for any category $\scC$ we
have the category $\FinSet/\scC$. By definition its objects are
functors of the form $\tau:X\to \scC$, where $X$ is a finite set
(i.e., a finite discrete category).  Morphisms are strictly commuting
triangles of the form
\[
\xy
(-10, 8)*+{X} ="1"; 
(10, 8)*+{Y} ="2";
(0,-4)*+{\scC}="3";
{\ar@{->}_{\tau} "1";"3"};
{\ar@{->}^{\varphi} "1";"2"};
{\ar@{->}^{\mu} "2";"3"};
\endxy , 
\]
where $\varphi$ is a map of finite sets.  We think of an object
$(X\xrightarrow{\tau}\scC)$ of $\FinSet/\scC$ as unordered list
$\{\tau(a)\}_{a\in X}$ of objects of the category $\scC$ indexed by
the finite set $X$. The composition of morphisms in $\FinSet/\scC$ is
given by pasting triangles together.
\begin{remark}
  In \cite{VSL} the category $\FinSet/\scC$ is called the category of
  typed finite sets of type $\scC$.
\end{remark}

\begin{remark} 
If the category $\scC$ has all finite products  there is a canonical functor 
\[
\PP= \PP_\scC:\left(\FinSet/\scC\right)^{\op} \to \scC.
\]
On objects the functor $\PP_\scC$  is defined by 
\[
\PP_\scC (\tau):= \prod _{x\in X}\tau(x).
\]
 On a morphism \[
\xy
(-10, 6)*+{X} ="1"; 
(10, 6)*+{Y} ="2";
(0,-6)*+{\scC}="3";
{\ar@{->}_{\tau} "1";"3"};
{\ar@{->}^{\varphi} "1";"2"};
{\ar@{->}^{\mu} "2";"3"};
\endxy 
\] 
the functor $\PP_\scC$ is defined by requiring that  
 the diagram
\[
\begin{tikzcd}
	\PP_\scC(\mu)\ar[r,dashed,"\PP_\scC(\varphi)"]\ar[d,"\pi_{\varphi(a)}"']
&
	\PP_\scC(\tau)\ar[d,"\pi_a"]
\\
	\mu(\varphi(a))\ar[r,"\id"']&\tau(a)
\end{tikzcd}
\]
commutes for all $a\in X$.  Here and elsewhere $\pi_a:\PP_\scC (\tau)=
\prod _{x\in X}\tau(x) \to \tau (a)$ is the projection on the $a$th
factor (i.e., one of the structure maps of the categorical product)
and $\pi_{\varphi(a)} $ is defined similarly.
\end{remark}
Since the objects of the category $\FinSet/\scC$ are functors it is
natural to modify the morphisms by allowing the triangles to be
2-commutative rather than strictly commutative. There are two choices
for the direction of the 2-arrow. If $\scC $ has finite products it
is natural to choose 2-commuting triangles of the form
\begin{equation}\label{eq:4.1}
\xy
(-10, 10)*+{X} ="1"; 
(10, 10)*+{Y} ="2";
(0,-2)*+{\scC }="3";
{\ar@{->}_{\tau} "1";"3"};
{\ar@{->}^{\varphi} "1";"2"};
{\ar@{->}^{\mu} "2";"3"};
{\ar@{=>}_{\scriptstyle \Phi} (4,6)*{};(-0.4,4)*{}} ; 
\endxy 
\end{equation}
as morphisms between lists. We denote this variant of
$\FinSet/\scC$ by $(\FinSet/\scC)^\Leftarrow$. The reason why this is
``natural'' is that if $\scC$ has finite products we again have a
 functor
\[
\PP = \PP_\scC: ((\FinSet/\scC)^\Leftarrow)^{\op} \to \scC.
\]
which is defined on objects by 
\[
\PP_\scC(\tau) := \prod_{a\in X} \tau(a).
\]
To extend the definition of $\PP_\scC$ to morphisms, we define
\[
\PP_\scC(\varphi, \Phi):= \PP_\scC\left(
\parbox{.8in}{
\xy
(-10, 10)*+{X} ="1"; 
(10, 10)*+{Y} ="2";
(0,-2)*+{\scC }="3";
{\ar@{->}_{\tau} "1";"3"};
{\ar@{->}^{\varphi} "1";"2"};
{\ar@{->}^{\mu} "2";"3"};
{\ar@{=>}_{\scriptstyle \Phi} (4,6)*{};(-0.4,4)*{}} ; 
\endxy 
}
\right)
\]
by requiring that the diagram
\begin{equation}\label{eq:*3}
\begin{tikzcd}[column sep=large]
	\PP_\scC(\mu)\ar[r,dashed,"{\PP_\scC(\varphi,\Phi)}"]\ar[d,"\pi_{\varphi(a)}"']
&
	\PP_\scC(\tau)\ar[d,"\pi_a"]
\\
	\mu(\varphi(a))\ar[r,"\Phi(a)"']
&
	\tau(a)
\end{tikzcd}
\end{equation}
commutes for all $a\in X$.

Here is an example of $(\FinSet/\scC)^\Leftarrow$ that we very much
care about: take $\scC = \SSub$, the category of surjective
submersions. Since $\SSub$ has finite products we have a contravariant
functor
\[
\PP_\SSub:((\FinSet/\SSub)^\Leftarrow)^{\op} \to \SSub.
\]
Other categories of interest are $\scC = \Man$, the category of
manifolds and $\scC = \Vect$, the category of vector spaces and linear
maps.  \mbox{}

\begin{remark} There is a canonical faithful functor $\jmath:
  \FinSet/\scC \hookrightarrow (\FinSet/\scC)^\Leftarrow$ which is
  identity on objects. The functor $\jmath$ sends an arrow
  $\varphi:(X\xrightarrow{\tau}\scC)\to (Y\xrightarrow{\mu}\scC)$ in
  $\FinSet/\SSub$ to the arrow
\[
\xy
(-10, 10)*+{X} ="1"; 
(10, 10)*+{Y} ="2";
(0,-2)*+{\scC }="3";
{\ar@{->}_{\tau} "1";"3"};
{\ar@{->}^{\varphi} "1";"2"};
{\ar@{->}^{\mu} "2";"3"};
{\ar@{=>}_{\scriptstyle id_{\tau} }(4,6)*{};(-0.4,4)*{}} ; 
\endxy 
\]
in $(\FinSet/\SSub)^\Leftarrow$ (note that $\tau = \mu\circ \varphi$).
If the category $\scC$ has finite products, the diagram
\[
\xy
(-20, 10)*+{\FinSet/\scC} ="1"; 
(20, 10)*+{(\FinSet/\scC)^\Leftarrow} ="2";
(0,-2)*+{\scC }="3";
{\ar@{->}_{\PP} "1";"3"};
{\ar@{->}^{\jmath} "1";"2"};
{\ar@{->}^{\PP} "2";"3"};
\endxy
\]
commutes.
\end{remark}
\begin{remark}
Note also that given two objects $X\xrightarrow{\tau}\scC$ and
$Y\xrightarrow{\mu}\scC$ in $\FinSet/\scC$ (or in
$(\FinSet/\scC)^\Leftarrow$) we have the new object
\[
\tau\sqcup \mu:X\sqcup Y\to \scC ,
\]
which is defined by taking disjoint unions (i.e., coproducts).
Moreover 
\[
\PP_\scC(\tau\sqcup \mu) = \PP_\scC(\tau)\times \PP_\scC (\mu),
\]
where $ \PP_\scC(\tau)\times \PP_\scC (\mu)$ is the categorical
product of $ \PP_\scC(\tau)$ and $\PP_\scC(\mu)$ in $\scC$.
\end{remark}

\subsection{ The functor $\odot:  ((\FinSet/\LinRel)^\Leftarrow)^{\op}
  \to \LinRel$} \label{subsec:4.2}
\mbox{}\\[4pt]
If $\scC = \LinRel$, the (2-)category of vector spaces and linear
relations, then the  category $(\FinSet/\LinRel)^\Leftarrow$ of finite unordered lists of vector spaces still makes
sense. However the existence of an extension of the assignment
\[
(\FinSet/\LinRel) \ni \tau\mapsto \oplus_{a\in X} \tau(a) \in \LinRel
\]
to a functor is a bit more delicate since the direct sum $\oplus$ is
not a product in $\LinRel$. This said, given a finite list $\tau: X\to
\LinRel$ of vector spaces and an arbitrary vector space $Z$ there is
a canonical map
\[
\bigoplus_{a\in X} \Hom_\LinRel (Z, \tau(a)) \to \Hom_\LinRel (Z, \oplus_{a\in X}
\tau(a)).
\]
It assigns to a collection of the subspaces $\{R_a\subset
\tau(a)\times Z\}_{a\in X}$ (that is, to a collection $\{R_a:
Z\tickar \tau(a)\}_{a\in X}$ of arrows in $\LinRel$) 
the intersection
\begin{multline} \label{eq:5**}
\bigcap _{a\in X} (\pi_a
\times\id_Z)\inv (R_a)= \\
\left\{((v_a)_{a\in X}, z)\in (\bigoplus_{a\in X} \tau(a)) \times Z\,\,
\middle|\,\, (v_a,z)\in R_a\textrm{ for all }a\in X\right\}.
\end{multline}
Here $\pi_a: \oplus_{x\in X} \tau(x) \to \tau(a)$ are the canonical
projections.
\begin{proposition}\label{prop:4.2}
The assignment 
\[
(\FinSet/\LinRel) \ni \tau \mapsto \oplus_{a\in X} \tau(a) \in \LinRel
\]
extends to a lax  functor
\[
\odot: ((\FinSet/\LinRel)^\Leftarrow)^{\op}   \to \LinRel.
\]
\end{proposition}
\begin{proof}
Given a list $\tau\in  (\FinSet/\LinRel)^\Leftarrow$ we define 
\[
\odot (\tau):= \oplus_{a\in X}\tau(a).
\]
Given a 2-commuting triangle 
\[
\xy
(-10, 10)*+{X} ="1"; 
(10, 10)*+{Y} ="2";
(0,-2)*+{\LinRel }="3";
{\ar@{->}_{\tau} "1";"3"};
{\ar@{->}^{\varphi} "1";"2"};
{\ar@{->}^{\mu} "2";"3"};
{\ar@{=>}_{\scriptstyle \Phi} (4,6)*{};(-0.4,4)*{}} ; 
\endxy 
\]
we set
\begin{equation}\label{eq:5*}
\odot (\varphi, \Phi):= \bigcap _{a\in X} (\pi_a\times
\pi_{\varphi(a)})\inv (\Phi(a)),
\end{equation}
where 
\[
\pi_a\times \pi_{\varphi(a)}: \oplus_{x\in X}\tau(x) \times \oplus
_{b\in Y} \mu(b) \to \tau(a)\times \mu(\varphi(a)) 
\]
are the
projections and the relations $\Phi(a):\mu(\varphi(a)\tickar \tau (a))$ are
the component relations of the natural transformation $\Phi: \mu\circ
\varphi \Rightarrow \tau$.
It remains to check that given a pair of triangles 
\[
\xy
(-20, 10)*+{X} ="-1"; 
( 0, 10)*+{Y} ="1"; 
(20, 10)*+{Z} ="2";
(0,-6)*+{\LinRel }="3";
{\ar@{->}_{\tau} "-1";"3"};
{\ar@{->}^{\varphi} "-1";"1"}; 
{\ar@{->}_{\mu} "1";"3"};
{\ar@{->}^{\psi} "1";"2"};
{\ar@{->}^{\nu} "2";"3"};
{\ar@{=>}_<<<{\scriptstyle \Psi} (10,7)*{};(6,4)*{}} ; 
{\ar@{=>}_<<<{\scriptstyle \Phi} (-4,7)*{};(-8,4)*{}} ; 
\endxy
\]
that can be composed (pasted together) we have
\[
\odot (\varphi\psi , \Phi\circ (\Psi \varphi)) \supseteq 
\odot (\varphi, \Phi)\circ \odot (\psi, \Psi). 
\]
This is a computation. By definition (see \eqref{eq:5**} and \eqref{eq:5*})   we have
\begin{align*}
\odot (\psi, \Psi)&= \{((w_b)_{b\in Y},(v_c)_{c \in Z })\mid (w_b,v_{\psi(b)} )\in \Psi (b) \textrm{ for all } b\in Y\}\\
\odot (\varphi, \Phi)&= \{((u_a)_{a\in X},(w_b)_{b \in Y})\mid (u_a,w_{\varphi(a)} ) \in  \Phi(a)\textrm{ for all } a\in X\}.
\end{align*}
Hence 
\begin{gather}\label{eq4.1}
  \odot (\varphi, \Phi)\circ \odot (q, \Psi) =
  \left\{((u_a)_{a\in X}, (v_c)_{c\in Z})\;\;\middle|\;\;
  \parbox{2in}{\raggedright $\exists (w_b)\in \oplus_{b\in Y} \mu(b)$ so that $\forall a,   
(u_a, w_{p(a)})\in \Phi(a)$ and $\forall b, (w_b, v_{q(b)})\in \Psi(b)$}\right\}.
\end{gather}
The left hand side, on the other hand, is a
subspace of 
\[
\odot (\tau) \times \odot (\nu)= \oplus _{a\in X} \tau (a)
\times \oplus_{c\in Z} \nu(c)
\]
which is  given by
\begin{align}\label{eq4.2}
\odot (\varphi\psi,&\Phi \circ (\Psi\varphi)):=\\ \nonumber
& 
\{((u_a)_{a\in X},
(v_c)_{c\in Z})\mid (u_a, v_{\psi(\varphi(a))}) \in \Phi(a)\circ
  \Psi(\varphi(a)) \textrm{ for all } a\in X\}.
\end{align}
Since the right hand side of \eqref{eq4.1} is contained in the right
hand side of \eqref{eq4.2}, the result follows.
\end{proof}

\begin{example}
  We use the notation of Proposition~\ref{prop:4.2} above.  Suppose that
  $X=\{1,2,3\}$, $\tau: X\to \LinRel$ is a function that assigns to
  each $i\in X$ the same vector space $W$, $Y$ is a one point set
  $\{*\}$, $\mu(*)$ is some vector space $V$, and  $\Phi(i) \subseteq
  \tau(i)\times \mu(*) = W\times V$, $i=1,2,3$, are some linear relations.  Then
  $\odot(\varphi, \Phi) \subseteq W^3 \times V$ is the relation
\[
\odot(\varphi, \Phi) =\{(w_1,w_2,w_3, v)\mid (w_i, v) \in \Phi(i),
\quad i=1,2,3\}.
\]
\mbox{} \hfill $\Box$
\end{example}

The following lemma will prove useful in  computing examples.
\begin{lemma}\label{lem:6.2}
If the components $\Phi(a):\mu(\varphi(a)) \tickar \tau(a)$ of the natural transformation $\Phi: \mu\circ \varphi \Rightarrow \tau$ in the morphism 
\[
\xy
(-10, 10)*+{X} ="1"; 
(10, 10)*+{Y} ="2";
(0,-2)*+{\LinRel }="3";
{\ar@{->}_{\tau} "1";"3"};
{\ar@{->}^{\varphi} "1";"2"};
{\ar@{->}^{\mu} "2";"3"};
{\ar@{=>}_{\scriptstyle \Phi} (4,6)*{};(-0.4,4)*{}} ; 
\endxy  
\]
in the category $\FinSet/\LinRel)^\Leftarrow$ are graphs of linear
maps $\phi(a):\mu(\varphi(a)) \to \tau(a)$, i.e., $\Phi(a) =
\Graph(\phi(a))$, then the relation $\odot(\varphi,\Phi): \odot(\mu)\tickar
\odot(\tau)$ is the graph of a linear map.
\end{lemma}
\begin{proof}
The linear map 
\[ 
\oplus(\varphi,\Phi):\oplus_{y\in Y}\mu(y) \to \oplus_{x\in X}\tau(x)
\] 
in question is uniquely defined by requiring that the diagrams
\[
\xy
(-20, 6)*+{\bigoplus_{y\in Y}\,\mu(y)}="1"; 
(10, 6)*+{\bigoplus_{x\in X}\,\tau(x)} ="2"; 
(-20, -6)*+{\mu(\varphi(a))}="3"; 
(10,-6)*+{\tau(a)}="4"; 
 {\ar@{->}_{\pi_{\varphi(a)}} "1";"3"}; 
{\ar@{-->}^{\oplus(\varphi,\Phi)} "1";"2"}; 
{\ar@{->}^{\pi_a} "2";"4"}; 
 {\ar@{->}_{\phi(a)} "3";"4"}; 
\endxy
\]
commute for all $a\in X$. %
The map  $\oplus(\varphi,\Phi)$ is the required map since
\begin{multline}
\Graph(\oplus (\varphi,\Phi)) =\\
\{((v_x), (w_y)) \in (\oplus_{x\in X} \tau(x)) \times (\oplus_{y\in Y} \mu(y))\mid
  w_{\varpi(a)} = \phi(a)(v_a) \textrm{ for all }a\in X\}\\
=\{((v_x), (w_y)) \mid  (v_a, w_{\varphi(a)})  \in \Graph(\phi(a))  = \Phi(a)\}\\
= \odot (\varphi, \Phi).
\end{multline}
\end{proof}

\subsection{The algebra $\cO\Crl: \cO((\SSubi)^\op)\to \cO\Vect$ in
  terms of lists}\label{subsec:5.2}\mbox{}\\[4pt]
In discussing colored operads in Section~\ref{sec:operad} we swept a
few details under the rug.  We now revisit the discussion and revise
Remark~\ref{rmrk:net_pattern}.

Given a symmetric monoidal category
$(\sfC, \otimes)$ a morphism in the corresponding colored operad
$\cO\sfC$ has more generally as its source a finite list
$X\xrightarrow{\tau} \sfC$.  A morphism in $\cO\sfC$ with the source
$X\xrightarrow{\tau} \sfC$ and target $b\in \sfC$ is a morphism in $\sfC$
from the product $x_1\otimes (x_2\otimes (\cdots \otimes x_n)\ldots)$
for some choice of a bijection $\nu:\{1,\ldots, n\}\to X$ (where we
set $x_i:= \tau (\nu(i))$).  We make such a choice for each list in
$\FinSet/\sfC$.  The choices don't matter thanks to Mac Lane's
coherence theorem. 

Since the functor $\Crl: ((\SSubi)^\op, \times) \to (\Vect, \oplus) $
is lax monoidal (see Lemma~\ref{lem:crl_lax}) for any {\bf ordered} list
$a:\{1,\ldots, n\}\to \SSub$ of submersions we have a canonical linear
map
\[
\Crl_a: \bigoplus_{i=1}^n \Crl(a(i))\to \Crl (\prod _{i=1}^n a(i)).
\]
It is given by 
\begin{equation}\label{eq:5.2}
(\psi_1,\ldots, \psi_n)\mapsto \psi_1\times \cdots \times \psi_n.
\end{equation}

\begin{lemma}\label{lem:5.14} 
For any unordered  list $\tau: X\to \SSub$ we have a canonical
linear map
\begin{equation} \label{eq:5.16}
 \Crl_\tau: \bigoplus _{x\in X}\Crl(\tau(x))\to \Crl(\prod_{x\in X} \tau(x)) 
\end{equation}
so that if  $X=\{1,\ldots,n\}$ then $\Crl_\tau$ is given by  \eqref{eq:5.2}.
\end{lemma}
\begin{proof}
We start by introducing notation. For each $a\in X$ we have a submersion $\tau(a)$ which we
  write as
\[
\tau=  (\tau(a)_{\tot}\xrightarrow{\upsilon_a} \tau(a)_{\st}).
\]
We denote the projections from $\bigoplus_{x\in X} \Control(\tau(x))$
to $\Control (\tau(a))$ by $\varpi_a$:
\begin{equation}\label{eq:5.15'}
\varpi_a:  \bigoplus_{x\in X} \Control(\tau(x)) \to \Control(\tau(a)).
\end{equation}
Since $\PP_\SSub(\tau) =\prod_{a\in X}\tau(a)$ we have canonical
projections
\[
\pi_a:  \prod_{x\in X}\tau(x) \to \tau(a). 
\]
Each $\pi_a$ is a pair of maps of manifolds:
\[
\pi_a^{\tot}:  \prod_{x\in X}\tau(x)_{\tot}  \to \tau(a)_{\tot}
\qquad\textrm{and}\qquad 
\pi_a^{\st}:  \prod_{x\in X}\tau(x)_{\st}  \to \tau(a)_{\st}.
\]
We set
\[
\vartheta_a := 
\upsilon_a \circ \pi_a^{\tot}:  \prod_{x\in X}\tau(x)_{\tot} \to \tau(a)_{\st}.
\]
We define the projections
\[
\Xi_a: \Control(\PP_\SSub(\tau))= 
\Control\left(\prod_{x\in    X}\tau(x)\right) \to 
\Control\left(\prod_{x\in    X}\tau(x)_{\tot}\xrightarrow{\vartheta_a} \tau(a)_{\st}\right)
\]
by 
\[
\Xi_a(\psi) := T\pi_a ^{\st} \circ \psi %
\]
for all $\psi\in \Control(\PP_\SSub(\tau))$. Since $T (\prod_{x\in
  X}\tau(x)_{\st}) = \prod_{x\in X}T\, \tau(x)_{\st}$, the
projections $\Xi_a$ make $\Control(\PP_\SSub(\tau))$ into a direct sum:
\[
\Control(\PP_\SSub(\tau)) = \bigoplus_{a\in X}\Control \left(\prod_{x\in
  X}\tau(x)_{\tot}\xrightarrow{\vartheta_a} \tau(a)_{\st}\right).
\]
Finally we have pull-back maps
\[
(\pi_a^{\tot})^*:  
\Control\left(\tau(a)_{\tot}\xrightarrow{\upsilon_a} \tau(a)_{\st}\right)
\to \Control\left(\prod_{x\in X}\tau(x)_{\tot}\xrightarrow{\vartheta_a} \tau(a)_{\st}\right),
\]
\[
(\pi_a^{\tot})^*\psi_a := \psi_a \circ \pi_a^{\tot}
\]
or all $\psi_a \in \Control(\tau(a)_{\tot}\xrightarrow{\upsilon_a} \tau(a)_{\st})$.
By the universal property of products the family of maps 
\[
\left\{
(\pi_a^{\tot})^* \circ \varpi_a:  \oplus_{x\in X} \Control(\tau(x)) \to
\Control\left(\prod_{x\in X}\tau(x)_{\tot}\xrightarrow{\vartheta_a} \tau(a)_{\st}\right)
\right\}
\]
uniquely define a linear map $\Control_{\tau}$ making the diagram
\begin{equation}\label{eq:4.7}
\xy
(-30, 10)*+{\bigoplus_{x\in X} \Control(\tau(x)) } ="1"; 
(30, 10)*+{\Control(\PP_\SSub(\tau)) } ="2"; 
(-30, -10)*+{ \Control(\tau(a)) } ="3";
(30, -10)*+{\Control\left(\prod_{x\in X}\tau(x)_{\tot}
\xrightarrow{\vartheta_a} \tau(a)_{\st}\right) } ="4"; 
{\ar@{-->}^{\Control_{\tau}} "1";"2"};
{\ar@{->}_{\varpi_a} "1";"3"};
{\ar@{->}_{\Xi_a} "2";"4"};
{\ar@{->}_{(\pi_a^{\tot})^*\qquad} "3";"4"};
\endxy
\end{equation}
of vector spaces and linear maps commute. 
Note that by definition of $\Crl_\tau$
\begin{equation}\label{eq:5.17}
T\pi_a^\st \circ \Crl_\tau((F_x)_{x\in X}) = F_a \circ \pi_a^\tot
\end{equation}
for any $(F_x)_{x\in X} \in \bigoplus_{x\in X} \Control(\tau(x)) $ and
any $a\in X$.
\end{proof}

\begin{remark}
  It follows from Lemma~\ref{lem:5.14} that given a list
  $X\xrightarrow{\tau} \SSub$ and an interconnection morphism
  $b\xrightarrow{\psi} \prod_{x\in X} \tau(x) = \PP _\SSub (\tau)$ we
  have a linear map
\[
\cO\Crl(\psi): \bigoplus_{x\in X} \Crl(\tau(x))\to \Crl(b),
\]
which is given by
\[
\cO\Crl(\psi)= \psi^* \circ \Crl_\tau.
\]
\end{remark}

\section{Wiring diagrams}\label{sec:WD}
In section~\ref{sec:operad} we constructed the colored operad
$\cO((\SSubi)^\op)$ and the algebra 
\[
\cO\Crl: \cO((\SSubi)^\op)\to \cO\Vect.
\]
This algebra is similar to the algebra 
\[
\cO\mcG: \cO \mathbf{W} \to\cO \Set
\]
over the operad of $\cO \mathbf{W}$ of wiring diagrams defined in \cite{VSL}.
We now contrast  and compare the two operads and the two algebras.

To make the comparison
easier we recall the definition of the monoidal category $\mathbf{W}$
of wiring diagrams and the functor $\mcG: {\bf W}\to \Set$.  Note first that
in \cite{VSL} open continuous time dynamical systems are viewed 
differently from the way we have been viewing them in this
paper.  There an open system consists of three manifolds $M, U^\inp,
U^\out$, a smooth map $f^\out: M\to U^\out$  and an open system 
$f^\inp\in \Crl( M\times U^\inp \to M)$.
To distinguish the two approaches we will refer to the tuple 
\[
(M,
U^\inp, U^\out, f^\inp, f^\out)
\]
as a {\sf factorized open system} $f^\inp$ with {\sf output}
$f^\out$. The manifolds $M, U^\inp, U^\out$ are, respectively, the
spaces of states, inputs and outputs of the system $(M, U^\inp,
U^\out, f^\inp, f^\out)$.  Factorized open systems with outputs form a
category, which in \cite{VSL} is called $\mathsf{ODS}$ (for {\bf O}pen
{\bf D}ynamical {\bf S}ystems).  By definition a morphism from $(M_1,
U_1^\inp, U_1^\out, f_1^\inp, f_1^\out)$ to $(M_2, U^\inp_2, U^\out_2,
f^\inp_2, f^\out_2)$ is a triple of maps $\zeta = (\zeta^\st: M_1\to
M_2, \zeta^\inp: U_1 ^\inp \to U_2 ^\inp, \zeta^\out: U_1 ^\out \to
U_2 ^\out)$ so that the following diagram
\[
\xymatrix{
M_1\times U_1^\inp \quad\ar[r]^{ f_1^\inp \times f_1^\out } \ar[d]_{\zeta^\st\times\zeta^\inp}
&\quad TM_1\times U_1^\out  \ar[d]^{T\zeta^\st\times\zeta^\out} 
\\
M_2\times U_2^\inp \ar[r]_{f_2^\inp \times f_2^\out} & TM_2\times U_2^\out }
\]
commutes.  This category has finite products.  The symmetric monoidal
category $\mathbf{W}$ is defined as follows.  The objects of $\bfW$
are pairs of unordered lists of manifolds (or, equivalently, pairs of
typed finite sets of type ``manifold'').  Thus by definition an object $X$ of $\bfW$
is an ordered pair $ (\tau^\inp: X^\inp\to \Man, \tau^\out:X^\out\to
\Man)$ of objects of $\FinSet/\Man$.  The objects of $\bfW$ are called
{\sf boxes}.  The morphisms in $\bfW$ are called {\sf wiring
  diagrams}.  A wiring diagram is a triple $(X,Y,\varphi)$ where $X,Y$
are boxes and
\[
\varphi: X^\inp \sqcup Y^\out \to X^\out\sqcup Y^\inp
\]
is an isomorphism in $\FinSet/\Man$ (here and below we are suppressing
maps to $\Man$) so that
\begin{equation}\label{eq:6.*}
\varphi(Y^\out) \subseteq X^\out.
\end{equation}
Condition \eqref{eq:6.*} allows us to decompose $\varphi$ into a pair
$\varphi=(\varphi^\inp,\varphi^\out)$:
\begin{align}\label{dia:components of wd}
   \left\{
     \begin{array}{l}
       \varphi^\inp: X^\inp \to X^\out\sqcup Y^\inp \\
      \varphi^\out: Y^\out  \to X^\out
     \end{array}
   \right. .
\end{align}
Defining composition of morphisms in $\mathbf{W}$ and proving that
composition is associative, that is, proving that $\mathbf{W}$ is
actually a category, requires work (see \cite{VSL}).  Compare that
with the construction of the category $\SSubi$.

A {\sf wire} in a wiring diagram $(X,Y,\varphi)$ is a pair $(a,b)$,
where $a\in X^\inp\sqcup Y^\out$, $b\in X^\out\sqcup Y^\inp$, and
$\varphi(a)=b$.  The monoidal product on $\bfW$ is disjoint union:
\begin{gather*}
  (\tau^\inp:X^\inp\to \Man, \tau^\out:X^\out\to \Man)\sqcup
  (\mu^\inp:Y^\inp\to \Man, \mu^\out:Y^\out\to \Man):= \mbox{
    \hspace{6cm}}
  \\
\mbox{\hspace{2.6cm}}
(\tau^\inp\sqcup \mu^\inp:X^\inp\sqcup Y^\inp \to \Man,
  \tau^\out\sqcup \mu^\out :X^\out\sqcup Y^\out\to \Man).
\end{gather*}
The semantics of $\bfW$ is obtained by filling in the boxes in the
following sense.  Given a box $X= (X^\inp, X^\out) \in
(\FinSet/\Man)^2$ we have a pair of manifolds $(\PP X^\inp, \PP
X^\out)$, where as before the functor $\PP= \PP_\Man:(\FinSet/\Man)^\op \to
\Man$ is defined on objects by taking products:
\[
\PP (\tau) := \prod _{a\in X}\tau(a).
\]
Therefore a choice of a manifold $M$ defines a product fiber bundle 
\[
M\times \PP X^\inp\to M.
\]
We then can further choose an output map $f^\out: M\to \PP X^\out$ and
a factorized open system $f^\inp: M\times \PP X^\inp \to TM$.  This is
the consideration behind the definition of the functor $\mcG: \bfW\to
\Set$.  Its value on an object $X$ of $\bfW$ is, by definition, the
collection
\[
\mcG (X):= \{(S, f)\mid S\in \FinSet/\Man, f^\inp\times
f^\out:\PP S \times \PP X^\inp \to T\PP S \times \PP X^\out\}
\]
where $f= f^\inp\times f^\out$ are factorized open systems with
outputs.  (To make sure that $\mcG (X)$ is actually a set and not a
bigger collection we should, strictly speaking, replace the category
$\Man$ of manifolds by an equivalent small category. For example we
can redefine $\Man$ to consist of manifolds that are embedded in the
disjoint union $\sqcup _{n\in \N} \R^n$.)

We now see that the monoidal category $\bfW$ is set up so that an
object is a kind of black box with wires sticking out. The wires are
partitioned into two sets.  The first set of wires receive inputs.
The other set of wires report outputs.  The box is filled with open
dynamical systems.  By design we have no direct access to the state
spaces of these systems.  Compare this with the category of $\SSub$
where the objects specify the spaces of states of the systems. 

Note also that the functor $\mcG: \bfW \to \Set$ is very coarse.
For example if we start with a box $X$ whose inputs $X^\inp$ and
$X^\out$ are both singletons $\{*\}$ and $\tau^\inp(*)= \tau^\out(*)=
$ a point then $\mcG(X)$ is (in bijective correspondence with) the set
of {\em all} possible continuous time closed dynamical systems.

To conclude our discussion of wiring diagrams, the differences between
the algebra $ \cO\Crl: \cO((\SSubi)^\op)\to \cO\Vect$ of this paper
and $\cO\mcG: \cO \mathbf{W} \to\cO \Set$ of \cite{VSL} are
differences in philosophy and in intended applications.  The approach
of \cite{VSL} is to treat an open system as a black box with the space
of internal states as completely unknown with the algebra supplying
all possible choices of state spaces.  By contrast in this paper we
treat the space of internal states (and the total space) as known and
have the algebra supply the possible choices of open systems
(``dynamics'') that live on a given surjective submersion.

\section{Extension of the functor $\Crl$ to the category $(\FinSet/\SSub)^\Leftarrow$ of lists of submersions}\label{sec:7}

This section and the next one are technical.  The main result of this
section is Lemma~\ref{thm:5.9}. It will be reformulated as
Lemma~\ref{thm:5.9mrk2} in the next section once double categories are
introduced.  Lemma~\ref{thm:5.9} is, in effect, half the proof of the
main theorem of the paper, Theorem~\ref{thm:main}.

In Lemma~\ref{lem:5.14} we extended the object part of
the functor $\Crl:\SSub \to \LinRel$ to finite unordered lists of
submersions, which are objects of the category
$(\FinSet/\SSub)^\Leftarrow$.  We would like to extend $\Crl$ to maps
between lists, that is, to morphisms in the category
$(\FinSet/\SSub)^\Leftarrow$ of lists of submersions.  To this end
consider a map
\[
\xy
(-10, 10)*+{X} ="1"; 
(10, 10)*+{Y} ="2";
(0,-6)*+{\SSub }="3";
{\ar@{->}_{\tau} "1";"3"};
{\ar@{->}^{\varphi} "1";"2"};
{\ar@{->}^{\mu} "2";"3"};
{\ar@{=>}_<<<{\scriptstyle \Phi} (4,6)*{};(-2,2)*{}} ; 
\endxy 
\] 
between two lists of submersions.  We then have a pair of linear maps
\[
\Crl_\tau:\bigoplus _{x\in X}\Crl(\tau(x))\to \Crl(\PP_\SSub (\tau)),\qquad
\Crl_\mu:\bigoplus _{y\in Y}\Crl(\mu(y))\to \Crl(\PP_\SSub (\mu)).
\]%
We also have a map of submersions
\[
\PP_\SSub(\varphi,\Phi): \PP_\SSub (\mu)\to \PP_\SSub(\tau),
\]
hence a linear relation
\[
\Crl(\PP_\SSub(\varphi,\Phi)): \Crl(\PP_\SSub (\mu))\tickar \Crl(\PP_\SSub(\tau)).
\]
On the other hand we have a morphism of lists of vector spaces
\[
\xy
(-10, 10)*+{X} ="1"; 
(10, 10)*+{Y} ="2";
(0,-6)*+{\LinRel }="3";
{\ar@{->}_{\Crl\circ \tau} "1";"3"};
{\ar@{->}^{\varphi} "1";"2"};
{\ar@{->}^{\Crl\circ \mu} "2";"3"};
{\ar@{=>}_<<<{\scriptstyle \Crl \circ \Phi} (4,6)*{};(-2,2)*{}} ; 
\endxy ,
\] 
which give rise to a linear relation
\[
\odot(\varphi, \Crl \circ \Phi):  \bigoplus _{y\in Y}\Crl(\mu(y))
\tickar  \bigoplus _{x\in X}\Crl(\tau(x))
\]
(see subsection~\ref{subsec:4.2}).  These four maps  are related
in the following way.

\begin{lemma}\label{thm:5.9}
Suppose 
\[
\xy
(-10, 10)*+{X} ="1"; 
(10, 10)*+{Y} ="2";
(0,-6)*+{\SSub }="3";
{\ar@{->}_{\tau} "1";"3"};
{\ar@{->}^{\varphi} "1";"2"};
{\ar@{->}^{\mu} "2";"3"};
{\ar@{=>}_<<<{\scriptstyle \Phi} (4,6)*{};(-2,2)*{}} ; 
\endxy 
\] 
is a morphism in the category of lists of submersions.
The linear map
\[
\Crl_\tau\times \Crl _\mu:  \bigoplus _{x\in X}\Crl(\tau(x)) \times  \bigoplus _{y\in Y}\Crl(\mu(y)) \to \Crl(\PP_\SSub (\tau))\times \Crl(\PP_\SSub (\mu))
\]
sends the linear relation $\odot(\varphi, \Crl \circ \Phi)$ to a subspace of
the relation $\Crl(\PP_\SSub(\varphi,\Phi))$.  That is, for any pair 
\[
((F_x)_{x\in X}, (G_y)_{y\in Y})\in \odot(\varphi, \Crl \circ \Phi)
\subseteq \bigoplus _{x\in X} \Crl(\tau(x)) \times \bigoplus_{y\in
  Y}\Control(\mu(y))
\]
of lists of open systems, the open system $ \Crl_\mu((G_y))$ (see
\eqref{eq:5.16}) is $\PP(\varphi, \Phi)$ related to the open system
  $\Crl_\tau((F_x))$.
\end{lemma}
\begin{proof}
Recall that the map $\PP_\SSub(\varphi, \Phi): \PP_\SSub(\mu) \to \PP_\SSub(\tau)$ is defined by requiring that the diagrams
\begin{equation}\label{eq:7*3}
\begin{tikzcd}[column sep=large]
	\PP_\SSub(\mu)\ar[r,dashed,"{\PP_\SSub(\varphi,\Phi)}"]\ar[d,"\pi_{\varphi(a)}"']
&
	\PP_\SSub(\tau)\ar[d,"\pi_a"]
\\
	\mu(\varphi(a))\ar[r,"\Phi(a)"']
&
	\tau(a)
\end{tikzcd}
\end{equation}
commutes for all $a\in X$ (compare with \eqref{eq:*3}).
 By definition of the relation $\Crl(\PP(\varphi, \Phi))$, 
\[
(\Crl_\tau ((F_x)),
  \Crl_\mu ((G_y))) \in \Crl(\PP(\varphi, \Phi)) 
\]
if and only the two open systems are $\PP(\varphi, \Phi)$ related.
That is, if and only if the diagram
\begin{equation}\label{eq:*1}
\xy
(-20, 10)*+{\PP(\mu)_{\tot}}="1";
(20,10)*+{T(\PP( \mu)_{\st})}="2";
(-20, -10)*+{\PP(\tau)_{\tot}}="3"; 
(20, -10)*+{T(\PP(\tau)_{\st})}="4"; 
{\ar@{->}^{\Crl_{\mu}((G_{y})) }"1";"2"};
{\ar@{->}_{\Crl_\tau ((F_x))} "3";"4"};
{\ar@{->}^{T(\PP(\varphi, \Phi)_{\st})} "2";"4"};
{\ar@{->}_{\PP(\varphi, \Phi)_{\tot}} "1";"3"};
\endxy
\end{equation}
commutes. By definition of the relation $\odot(\varphi, \Phi):  \oplus_{y\in
    Y}\Control(\mu(y))\tickar \oplus _{x\in X} \Crl(\tau(x))$ a pair 
\[
((F_x)_{x\in X}, (G_y)_{y\in Y})\in  
\bigoplus _{a\in X} \Crl(\tau(a)) \times \bigoplus_{b\in    Y}\Control(\mu(b))
\]
of lists of open systems belongs to the relation $\odot(\varphi,
\Phi)$ if and only if $F_x$ and $G_{\varphi(x)}$ are $\Phi(x)$ related
for all $x\in X$.  That is, if and only if
\begin{equation}\label{eq:*0}
\xy
(-20, 10)*+{\mu(\varphi(x))_{\tot}}="1";
(20,10)*+{T\mu(\varphi(x))_{\st}}="2";
(-20, -10)*+{\tau(x)_{\tot}}="3"; 
(20, -10)*+{T\tau(x)_{\st}}="4"; 
{\ar@{->}^{G_{\varphi(x)} } "1";"2"};
{\ar@{->}_{F_x} "3";"4"};
{\ar@{->}^{T(\Phi(x)_{\st})} "2";"4"};
{\ar@{->}_{\Phi(x)_{\tot}} "1";"3"};
\endxy
\end{equation}
commutes for all $x\in X$.  The tangent bundle of a product of
manifolds is a product of tangent bundles
\[
T(\PP(\tau)_{\st}) = \prod_{a\in X}(T\tau(a)_{\st}).
\]  
Consequently
the diagram \eqref{eq:*1} commutes if and only if 
\begin{equation}\label{eq:*2}
T\pi_a^{\st} \circ \Crl_{\tau}((F_x))\circ \PP(\varphi, \Phi)_{\tot}=
 T\pi_a^{\st} \circ T(\PP(\varphi, \Phi)_{\st})\circ \Crl_{\mu}((G_y))
 \end{equation}
 for all $a\in X$. Here, as above, $\pi_a^{\st}: \prod_{x\in X}
 \tau(x)_{\st}\to \tau(a)_{\st}$ is the projection on the $a^{th}$
 factor.  By definition of $\Crl_\tau$  (see \eqref{eq:5.17})
\[
T\pi_a^\st \circ \Crl_\tau((F_x)_{x\in X}) = F_a \circ \pi_a^\tot
\]
for any $a\in X$.  By \eqref{eq:7*3}
\[
\pi_a^\tot \circ \PP (\varphi, \Phi)_\tot = 
     \Phi(a)_\tot \circ \pi^\tot_{\varphi(a)}
\]
Hence the left hand side of \eqref{eq:*2} is 
\[
F_a \circ \Phi(a)_{\tot} \circ \pi_{\varphi(a)}^{\tot}.
\]
We next compute the right hand side. Using the commutativity of
\eqref{eq:7*3} again we see that
\begin{eqnarray*}
\shortintertext{
\[
T \pi_a^{\st} \circ T(\PP(\varphi, \Phi)_{\st}) \circ \Crl_{\mu}((G_y))= 
\mbox{\hspace{16cm}}
\]
}
&=&T(\Phi(a)_{\st})\circ T\pi_{\varphi(a)}^{\st} \circ \Crl_{\mu} ((G_y))\\
&=& T(\Phi(a)_{\st})\circ (\pi_{\varphi(a)}^{\tot})^* \circ \varpi _{\varphi(a)} ((G_y))
\qquad  \mbox{\hspace{2cm}} \textrm{ (by definition of 
} \Crl_{\mu}, \textrm{ see \eqref{eq:5.17})}\\
&=&  T(\Phi(a)_{\st})\circ G_{\varphi(a)} \circ \pi_{\varphi(a)}^{\tot}\qquad  
\mbox{\hspace{.7cm}} \textrm{ (by definitions of $(\pi_{\varphi(a)}^{\tot})^*$ and of $\varpi _{\varphi(a)}$, see \eqref{eq:5.15'})}\\
&=& F_a \circ \Phi(a)_{\tot} \circ \pi_{\varphi(a)}^{\tot} \mbox{\hspace{6cm}}  
\qquad
\textrm{ (since \eqref{eq:*0} commutes).}
\end{eqnarray*}
Thus \eqref{eq:*2} holds.

\end{proof}
Lemma~\ref{thm:5.9} has a nice formulation in terms of double
categories --- see Lemma~\ref{thm:5.9mrk2}.  We will reformulate
Lemma~\ref{thm:5.9} after recalling the notions of a double category
in the next section.

\section{Double categories}\label{sec:double}

The goal of this section is to introduce double categories, to
reformulate Lemma~\ref{thm:5.9} in terms of double categories as
Lemma~\ref{thm:5.9mrk2}, and to prove
Lemma~\ref{lemma:crl_ssub_relvect}.  Lemmas~\ref{thm:5.9mrk2} and
\ref{lemma:crl_ssub_relvect} are used in the next section to prove
Theorem~\ref{thm:main}, which is the main theorem of the paper.

Recall that one can define (strict) double categories as
categories internal to the category $\mathsf{Cat}$ of categories.  

\begin{definition}\label{def:double_cat}
 A {\sf double category} $\D$ consists of two categories $\D_1$
(of arrows) and $\D_0$ (of objects) together with five structure
{\em functors}:
\begin{gather*}
  \frs, \frt:\D_1 \to \D_0 \quad (\textrm{source and target}), \qquad 
  \fru: \D_0 \to \D_1\quad (\textrm{unit})\\
  \fm: \D_1\times_{\frs, \D_0, \frt} \D_1\to \D_1 , 
\quad  (\xleftarrow{\mu}, \xleftarrow{\gamma}) \mapsto 
\fm(\xleftarrow{\mu}, \xleftarrow{\gamma})=: \mu* \gamma
\qquad (\textrm{multiplication/composition})
\end{gather*}
so that
\begin{gather*}
\frs \circ \fru = id_{\D_0}= \frt \circ \fru, \\
\frs (\mu *\gamma) = \frs (\gamma),\qquad \frt (\mu *\gamma) = \frt (\mu),
\end{gather*}
for all arrows $\mu,\gamma$ of $\D_1$ and
\begin{gather*}
(\mu * \gamma)* \nu =
\mu *(\gamma *\nu),\\
\mu* \fru(\frs(\mu)) = \mu,\qquad \fru(\frt(\mu))* \mu =\mu
\end{gather*}
for all arrows $\mu, \gamma, \nu$ of  $\D_1$.
\end{definition}
\begin{notation} We call the objects of $\D_0$ {\sf $0$-cells} or {\sf
    objects} and the morphisms of $\D_0$ the (vertical) {\sf
    1-morphisms}.  We depict 1-morphisms as $a\xleftarrow{f}b$. The
  objects of $\D_1$ are called the (horizontal) {\sf 1-cells} (or the
  horizontal {\sf 1-morphisms}).  A morphism $\alpha:\mu\to \nu$ of
  $\D_1$ with $\frs (\alpha) = f$ and $\frt(\alpha) =g$ is a {\sf
    2-morphism} or a {\sf 2-cell} (we will use the two terms
  interchangeably, {\em pace} Shulman \cite{Shulman10}).  We depict the two cell $\alpha$ 
  as
\begin{equation}\label{eq:2-mor}
\begin{tikzcd}
	c\ar[d,"g"']&a\ar[l,tick, "\mu"']\ar[d,"f"]\\
	d&b\ar[l,tick,"\nu"]\ar[lu,phantom,
        "{\Box}"]
\end{tikzcd}
.
\end{equation}
\end{notation}
\begin{remark}
A 2-morphism in a double category $\D$ of the form 
\[
\begin{tikzcd}
	c\ar[d,"id_c"']&a\ar[l,tick, "\mu"']\ar[d,"id_a"]\\
	c&a\ar[l,tick,"\nu"]\ar[lu,phantom,%
        "{\Box}"]
\end{tikzcd}
\]
 is called {\sf globular}.  Associated to every
strict double category $\D$ there is  a {\sf horizontal 2-category} $\scH(\D)$
consisting of objects, (horizontal) 1-cells and globular 2-morphisms.
\end{remark}

The main double category of interest for us is the double category
$\RelVect$ of vector spaces, linear maps and linear relations which we
presently define.

\begin{definition}[The double category $\RelVect$ of vector spaces,
  linear maps and linear relations]
  The category $\RelVect_0$ is the category of real vector spaces and
  linear maps.  The objects of the category $\RelVect_1$ are linear
  relations. Recall that we use the convention that a linear relation
  $R$ from a vector space $V$ to a vector space $W$ is a subspace of
  $W\times V$, which we write as $W\stackrel{R}{\ltickar}V$.  A
  morphism
\[
\alpha: (W\stackrel{R}{\ltickar}V) \to (Y\stackrel{S}{\ltickar}X)
\]
in $\RelVect_1$, that is, a 2-morphism in $\RelVect$, is a pair of linear
maps $f:V\to Y$, $g:W\to Y$ so that $(g\times f) (R)\subset S$ which
we picture as
\begin{equation}\label{eq:62.1}
\alpha =\quad
\begin{tikzcd}
	W\ar[d,"g"']&V\ar[l,tick, "R"']\ar[d,"f"]\\
	Y&X\ar[l,tick,"S"]\ar[lu,phantom,
        "{\Box}"]
\end{tikzcd}.
\end{equation}
Note that the 2-cell $\alpha$ in \eqref{eq:62.1} is completely
determined by its ``edges'' $g$, $f$, $R$ and $S$.

The composition of morphisms in $\RelVect_1$ is the compositions of
pairs of linear maps: given 
\[
\beta = \quad
\begin{tikzcd}	Y\ar[d,"k"']&X\ar[l,tick, "S"']\ar[d,"h"]\\
	U&Z\ar[l,tick,"T"]\ar[lu,phantom,
        "{\Box}"]
\end{tikzcd}
\quad
\textrm{and}  
\quad 
\alpha =\quad
\begin{tikzcd}
	W\ar[d,"g"']&V\ar[l,tick, "R"']\ar[d,"f"]\\
	Y&X\ar[l,tick,"S"]\ar[lu,phantom,
        "{\Box}"]
\end{tikzcd},
\]
\[
\beta \circ \alpha = \quad
\begin{tikzcd}
	W\ar[d,"kg"']&V\ar[l,tick, "R"']\ar[d,"hf"]\\
	U&Z\ar[l,tick,"T"]\ar[lu,phantom,
        "{\Box}"]
\end{tikzcd}.
\]

 Note that we {\sf do not} require the 2-morphisms in
$\RelVect$ to be commutative diagrams of relations: the condition
$(g,f)(R)\subset S$ does not imply that $\grph(f) \circ R = S \circ
\grph(f')$ (where $\grph(f)$, $\grph(g)$ are the graphs of $f$ and $g$
respectively and $\circ$ stands for the composition of relations).

To finish the description of the double category $\RelVect$ we need to
list the functors $\fru,\frs,\frt$ and $\fm$.  This is not
difficult. The functor $\fru$ sends a linear map $f:V\to W$ to the
2-morphism
\[
\fru(f): =\quad
\begin{tikzcd}
	V\ar[d,"f"']&V\ar[l,tick, "id_V"']\ar[d,"f"]\\
	W&X\ar[l,tick,"id_W"]\ar[lu,phantom,
        "{\Box}"]
\end{tikzcd}.
\]

The result of applying the functor $\frs$ to the 2-morphism \eqref{eq:62.1}
is the linear map $f$ and the result of applying the functor $\frt$ is
the linear map $g$.  The horizontal composition $\fm$ is given by
\[
\fm \left(
\begin{tikzcd}
	U\ar[d,"h"']&W\ar[l,tick, "R'"']\ar[d,"g"]\\
	Z&Y\ar[l,tick,"S'"]\ar[lu,phantom,
        "{\Box}"]
\end{tikzcd}\quad
,
\quad \begin{tikzcd}
	W\ar[d,"g"']&V\ar[l,tick, "R"']\ar[d,"f"]\\
	Y&X\ar[l,tick,"S"]\ar[lu,phantom,
        "{\Box}"]
\end{tikzcd} \right):=\quad
\begin{tikzcd}
	U\ar[d,"h"']&V\ar[l,tick, "R'\circ R"']\ar[d,"f"]\\
	Z&X\ar[l,tick,"S'\circ S"]\ar[lu,phantom,
        "{\Box}"]
\end{tikzcd}
\]
\end{definition}

\begin{remark}
  The horizontal category of the double category $\RelVect$ is the
  2-category $\LinRel$ (see Definition~\ref{def:lin-rel}):
\[
\scH(\RelVect) = \LinRel.
\]
\end{remark}

We can now restate Lemma~\ref{thm:5.9} in terms of the
double categories:
\begin{lemma}\label{thm:5.9mrk2}
A morphism 
\[
\xy
(-10, 10)*+{X} ="1"; 
(10, 10)*+{Y} ="2";
(0,-6)*+{\SSub }="3";
{\ar@{->}_{\tau} "1";"3"};
{\ar@{->}^{\varphi} "1";"2"};
{\ar@{->}^{\mu} "2";"3"};
{\ar@{=>}_<<<{\scriptstyle \Phi} (4,6)*{};(-2,2)*{}} ; 
\endxy 
\] 
in the category of lists of submersions gives rise to the 2-cell
\begin{equation}\label{eq:4.70}
\begin{tikzcd}[column sep=huge, row sep=large]
        \bigoplus_{y\in Y} \Control(\mu(y))
	\ar[r,tick,"{\odot(\varphi,\Control\circ \Phi)}"]
	\ar[d,"\Control_{\mu}"']
&
	 \bigoplus_{x\in X} \Control(\tau(x))
	\ar[d,"\Control_{\tau}"]
\\
	\Control(\PP(\mu))
	\ar[r,tick,"{\Control(\PP(\varphi, \Phi))}"']
&
	\Control(\PP(\tau))
	\ar[lu,phantom,description,"\Box"]
\end{tikzcd}
\end{equation}
in the double category $\RelVect$ of vector spaces, linear maps and
linear relations. %
\end{lemma}

Next we  clarify Remark~\ref{rmrk:SSub_as_double}.

\begin{definition}[The double category $\SSub^\Box$ of
  submersions]\label{def:ssub_box}
  The objects of the double category $\SSub^\Box$ are surjective
  submersions.  The (horizontal) 1-cells/1-morphisms are maps of
  submersions.  The vertical 1-morphisms are interconnection morphisms
  (see Definition~\ref{def:intercon}).  The 2-cells of $\SSub^\Box$
  are squares of the form
\[
\begin{tikzcd}
	c\ar[d,"g"']&a\ar[l,tick, "\mu"']\ar[d,"f"]\\
	d&b\ar[l,tick,"\nu"]\ar[lu,phantom,
        "{\Box}"]
\end{tikzcd}
\]
where $\mu,\nu$ are maps of submersions, $f$, $g$ are interconnection
morphisms and
\[
g\circ \mu = \nu\circ f
\]
in the category $\SSub$ of submersions.
\end{definition}

\begin{remark}
  Any category $\scC$ trivially defines a double category
  $\mathrm{db}\scC$: the 2-cells of $\mathrm{db}\scC$ are commuting
  squares in $\scC$.  We can also restrict the morphisms in these
  commuting squares by requiring that say vertical morphisms belong to
  a subcategory $\scD$ of $\scC$.  In Definition~\ref{def:ssub_box} above we
  took $\scC$ to be $\SSub$ and its subcategory $\scD$ to be $\SSubi$.
\end{remark}

It will be useful in understanding maps of networks to view the
functor $\Crl$ as a map of double categories.

\begin{lemma} \label{lemma:crl_ssub_relvect} 
A 2-cell
\begin{equation}\label{eq:square}
\begin{tikzcd}
	b&a\ar[l,tick, "\mu"']\\
	c\ar[u,"\psi"]&d\ar[u,"\varphi"']\ar[l,tick,"\nu"]
\ar[lu,phantom,       "{\Box}"]
\end{tikzcd}
\end{equation}
in the double category $\SSub^\Box$ of submersions (where $\mu, \nu$
are submersions and $\psi, \varphi$ are interconnection maps) gives
rise to the 2-cell
\[
\begin{tikzcd}
	\Crl(b)\ar[d,"\psi^*"']&\Crl(a)\ar[l,tick, "\Crl(\mu)"']
\ar[d,"\varphi^*"]\\
	\Crl(c)&\Crl(d)\ar[l,tick,"\Crl(\nu)"]
   \ar[lu,phantom,       "{\Box}"]
\end{tikzcd}
\]
in the double category $\RelVect$.  In other words if the open
systems $F\in \Crl(a)$ and $G\in \Crl(b)$ are $\mu$-related
then $\varphi^*F$ and $\psi^*G$ are $\nu$-related.
\end{lemma}
\begin{proof}
Since $F$ and $G$ are $\mu$-related  
\begin{equation}\label{eq:7-2}
T\mu_\st \circ F = G\circ \mu_\tot
\end{equation}
(see Notation~\ref{note:2.15}).  By definition of $\varphi^*$ (see
Remark~\ref{rmrk:3.6})
\begin{equation}\label{eq:7-3}
\varphi^*F = T(\varphi_\st)\inv \circ F \circ \varphi_\tot.
\end{equation}
Similarly
\begin{equation}\label{eq:7-4}
\psi^*G = T(\psi_\st)\inv \circ G \circ \psi_\tot.
\end{equation}
We compute:
\begin{eqnarray*}
  T \nu_{\st} \circ  \varphi^*F'&= 
T\nu_\st \circ T(\varphi_\st)\inv \circ F \circ \varphi_\tot&\qquad 
\textrm{ by \eqref{eq:7-3}}\\
  &= T(\psi_\st)\inv \circ T\mu_\st \circ F \circ \varphi_\tot
&\qquad \textrm{ by commutativity of \eqref{eq:square} and chain rule}
\\
  &= T(\psi_\st)\inv \circ G \circ \mu_\tot \circ \varphi_\tot
  &\qquad \textrm{ by \eqref{eq:7-2}}\\
  &= T(\psi_\st)\inv \circ G \circ \psi_\tot \circ \nu_\tot
&\qquad\textrm{ since \eqref{eq:square} commutes}\\
 &=\psi^*G \circ \nu_\tot & \qquad \textrm{ by \eqref{eq:7-4}.}
\end{eqnarray*}
\end{proof}

\section{Maps between networks of open systems}\label{sec:net_open}

Recall (Definition~\ref{def:network}) that a network of open
systems is an unordered list of submersions $\tau:X\to \SSub$ indexed
by a finite set $X$ together with an interconnection morphism
$\psi:b\to \prod_{x\in X} \tau(x)\equiv \PP_\SSub (\tau)$.

\begin{definition}[Maps between networks of open systems]
\label{def:map_of_networks} 
A {\sf map} from a network $(\tau:X\to \SSub, \psi:b\to\PP_\SSub
(\tau)\equiv \prod_{x\in X}\tau(x) )$ of open systems to a network
$(\mu:Y\to \SSub, \nu:c\to\PP_\SSub (\mu)\equiv \prod_{y\in Y}\mu(y))$
is a morphism
\[
\xy
(-10, 10)*+{X} ="1"; 
(10, 10)*+{Y} ="2";
(0,-2)*+{\SSub }="3";
{\ar@{->}_{\tau} "1";"3"};
{\ar@{->}^{\varphi} "1";"2"};
{\ar@{->}^{\mu} "2";"3"};
{\ar@{=>}_{\scriptstyle \Phi} (4,6)*{};(-0.4,4)*{}} ; 
\endxy 
\]
in the category $(\FinSet/\SSub)^\Leftarrow$ of lists of submersions
together with a map of submersions $f:c\tickar b$ so that the diagram
of submersions
\begin{equation}\label{eq:open-net}
\begin{tikzcd}[column sep=large, row sep=large]
\prod_{y\in Y}\mu(y)	\ar[r,tick,"{\PP_\SSub(\varphi, \Phi)}"]	&        \prod_{x\in X}\tau(x)\\
c 	\ar[u,"\nu"]	\ar[r,tick,"f"']& 
	b	\ar[u,"\psi"']
	\ar[lu,phantom,description,"\Box"]
\end{tikzcd}
\end{equation}
defines a 2-cell in the double category $\SSub^\Box$.  That is, we
require that
\[
\psi \circ f =  \PP_\SSub(\varphi, \Phi) \circ \nu.
\]
\end{definition}

\begin{example}\label{ex:9.1}  Here is an example of a map between two
 networks of open systems.  

  Our first network is the network $(\tau:X\to \SSub, \psi:b\to
  \prod_{x\in X}\tau(x))$ of Example~\ref{ex:3.13}, where $M, U$ are
  smooth manifolds, $p:M\times U\to M$ is a trivial fiber bundle,
  $\phi:M\to U$ a smooth map, $X=\{1, 2, 3\}$, $\tau:X\to \SSub$ is
  constant map with $\tau(1)= \tau(2)= \tau(3) = (p:M\times U\to M)$,
  $b$ is the trivial submersion $\id_{M^3}:M^3 \to M^3$ and
\[
\psi: b\to \prod_{x\in X}\tau(x) \simeq ((M\times U)^3 \to M^3)
\]
is the interconnection morphism with $\psi_\tot$ given by
\[
\psi_\tot(m_1, m_2, m_3)= ((m_1, \phi(m_2)), (m_2, \phi(m_1)),(m_3, \phi(m_2))).
\]
Our second network is $(\mu:Y\to \SSub,  \nu:c \to \prod_{y\in
  Y}\mu(y))$ where $Y=\{*\}$, $\mu(*) = (p:M\times U\to M)$, $c$ is
the trivial submersion $\id_M:M\to M$ and $\nu: c\to \mu(*)$ is the
interconnection morphism with $\nu_\tot$ given by
\[
\nu_\tot(m) = (m,\phi(m)).
\]
We now write down a map of networks.  We take $\varphi: X\to Y$ to be
the only possible map and set $\Phi_i:\mu(*)\to \tau(i)$ to be the
identity map for all $i\in X$.  That is, consider the morphism of
lists
\[
\xy
(-10, 8)*+{\{1,2,3\}} ="1"; 
(10, 8)*+{\{*\}} ="2";
(0,-4)*+{\SSub}="3";
{\ar@{->}_{\tau} "1";"3"};
{\ar@{->}^{\varphi} "1";"2"};
{\ar@{->}^{\mu} "2";"3"};{\ar@{=>}_{\scriptstyle \Phi} (4,5)*{};(-0.4,3)*{}} 
\endxy . 
\]
  In this case the induced map 
\[
\PP(\varphi, \Phi): (M\times U\to M) \to
  (M\times U\to M)^3
\]
 is the diagonal map
\[
\Delta: (M\times U\to M) \to (M\times U\to M)^3,
\qquad \Delta (m,u) = ((m,u), (m,u),(m,u)).
\]
We take $f:c\to b$ to be the diagonal map as well:
\[
f_\tot(m) = (m,m,m)
\]
Since 
\[
(\psi_\tot \circ f_\tot)\, (m) = ((m,\phi(m)), (m,\phi(m)),(m,\phi(m)))
= ((\PP_\SSub(\varphi, \Phi))_\tot \circ \nu_\tot)\, (m)
\]
for all $m\in M$,
 $\left(\xy
(-10, 6)*+{X} ="1"; 
(10, 6)*+{Y} ="2";
(0,-6)*+{\SSub }="3";
{\ar@{->}_{\tau} "1";"3"};
{\ar@{->}^{\varphi} "1";"2"};
{\ar@{->}^{\mu} "2";"3"};
{\ar@{=>}_{\scriptstyle \Phi} (4,2)*{};(-0.4,0)*{}} ; 
\endxy , f:c\tickar b\right)$ is a morphism of networks.  \\[12pt]

Note a curious feature of the example described above.  Consider an arbitrary
open system $F\in \Crl(M\times U\to M)= \Crl (\mu(*))$. The
interconnection map $\nu: (M\to M)\to (M\times U\to M)$ pulls $F$ back
to a vector field $v:M\to TM$ which is given by
\[
v(m) = F(m,\phi(m)).
\]
On the other hand
\[
F\times F\times F: (M\times U)^3 \to TM^3
\]
is an open system on $\prod_{x\in X}\tau(x) = (M\times U\to M)^3$ and
the interconnection map $\psi: b\to \prod_{x\in X}\tau(x)$ pulls
$F\times F \times F$ to a vector field $u:M^3 \to TM^3$.  This vector
field is given by
\begin{equation}\label{eq:vfu}
u(m_1, m_2, m_3)= (F(m_1,\phi(m_2)), F(m_2,\phi(m_1)),F(m_3,\phi(m_2)).
\end{equation}
Note that the diagonal 
\[
\Delta_{M^3}=\{(m_1, m_2, m_3) \in M^3 \mid m_1 = m_2 = m_3\}
\]
is an invariant submanifold of the vector field $u$.  We can rephrase
the invariance of the diagonal by saying that the vector fields $v$
and $u$ are $f:M\to M^3$-related, where as before $f$ denotes the
diagonal map.  Theorem~\ref{thm:main} below shows that the invariance
of the diagonal $\Delta_{M^3}$ for all vector fields on $M^3$ of the
form \eqref{eq:vfu} is not just a lucky accident.
\end{example}

\begin{theorem}\label{thm:main}
  A map 
\[
((\varphi,\Phi), f):(\tau:X\to \SSub,
  \psi:b\to\PP_\SSub (\tau)) \to (\mu:Y\to
  \SSub, \nu:c\to\PP_\SSub (\mu))
\]
of networks of open systems 
gives rise to a 2-cell
\begin{equation}\label{eq:9}
\begin{tikzcd}[column sep=large, row sep=large]
          \bigoplus_{y\in Y} \Control(\mu(y))
	\ar[r,tick,"{\odot(\varphi,\Control\circ \Phi)}"]
	\ar[d,"\nu^*\circ \Control_{\mu}"']
&
	 \bigoplus_{x\in X} \Control(\tau(x))
	\ar[d,"\psi^*\circ \Control_{\tau}"]
\\
	\Control(c)
	\ar[r,tick,"{\Control(f)}"']
&
	\Control(b)
	\ar[lu,phantom,description,"\Box"]
\end{tikzcd}
\end{equation}
in the double category $\RelVect$ of vector spaces, linear maps and
linear relations.
\end{theorem}

\begin{proof}
By Lemma~\ref{thm:5.9mrk2} the map \[
\xy
(-10, 10)*+{X} ="1"; 
(10, 10)*+{Y} ="2";
(0,-6)*+{\SSub }="3";
{\ar@{->}_{\tau} "1";"3"};
{\ar@{->}^{\varphi} "1";"2"};
{\ar@{->}^{\mu} "2";"3"};
{\ar@{=>}_<<<{\scriptstyle \Phi} (4,6)*{};(-2,2)*{}} ; 
\endxy 
\] 
in the category of lists of submersions gives rise to the 2-cell
\begin{equation}\label{eq:9.4}
\begin{tikzcd}[column sep=huge, row sep=large]
        \bigoplus_{y\in Y} \Control(\mu(y))
	\ar[r,tick,"{\odot(\varphi,\Control\circ \Phi)}"]
	\ar[d,"\Control_{\mu}"']
&
	 \bigoplus_{x\in X} \Control(\tau(x))
	\ar[d,"\Control_{\tau}"]
\\
	\Control(\prod_{y\in Y}\mu(y))
	\ar[r,tick,"{\Control(\PP_\SSub(\varphi, \Phi))}"']
&
	\Control(\prod_{x\in X}\tau(x))
	\ar[lu,phantom,description,"\Box"]
\end{tikzcd}
\end{equation}
in $\RelVect$.   By Lemma~\ref{lemma:crl_ssub_relvect} the 2-cell
\[
\begin{tikzcd}[column sep=large, row sep=large]
        \prod_{y\in Y}\mu(y)
	\ar[r,tick,"{\PP_\SSub(\varphi, \Phi)}"]
&
	\prod_{x\in X}\tau(x)	
\\
	 c	\ar[u,"\nu"'] 	\ar[r,tick,"f"']		
&
	b	\ar[u,"\psi"']	
	\ar[lu,phantom,description,"\Box"]
\end{tikzcd}
\]
in $\SSub^\Box$ gives rise to the 2-cell
\begin{equation}\label{eq:9.5}
\begin{tikzcd}[column sep=huge, row sep=large]
        \Crl(\prod_{y\in Y}\mu(y))\ar[d,"\nu^*"]
	\ar[r,tick,"{\Crl(\PP_\SSub(\varphi, \Phi))}"]
&
	\Crl(\prod_{x\in X}\tau(x))\ar[d,"\psi^*"]
\\
	 \Crl(c)		\ar[r,tick,"f"']		
&
	\Crl(b)	
\end{tikzcd}
\end{equation}
in $\RelVect$.  The vertical composite of the 2-cell \eqref{eq:9.4}
followed by the 2-cell \eqref{eq:9.5} is the 2-cell \eqref{eq:9}.
\end{proof}
We now illustrate Theorem~\ref{thm:main} in several more examples.
Before we do that We remind the reader that the space $$\Crl(\id: M\to
M)$$ of open systems on the identity map $\id: M\to M$ for some
manifold $M$ is the space of vector fields $\scX(M)$ on the manifold
$M$. %
\begin{remark}
Note  that a
map $f$ from a trivial fiber bundle $p: M\times U\to M$ to a trivial
fiber bundle $q: N\times W \to N$ is completely determined by the map
$f_{\tot}: M\times U \to N\times W$ between their total spaces (cf.\
Definition~\ref{def:mor_ssub}). Indeed, since the diagram 
\[
 \xy 
(-15,10)*+{M\times U}="1";
(15, 10)*+{N\times W}="2";
(-15,-5)*+{M}="3";
(15, -5)*+{N}="4";
{\ar@{->}^{f_{\tot}} "1";"2"};
{\ar@{->}_{} "1";"3"};
{\ar@{->}^{} "2";"4"};
{\ar@{->}_{f_{\st}} "3";"4"};
\endxy
\]
commutes, the map $f_\tot$ has to be of the form
\[
f_\tot(m,u) = (f_\st(m), h(m,u))
\]
 for some smooth map $h:M\times U\to W$.
\end{remark}

\begin{example}\label{ex:1.2}
Consider the map of lists 
\[
\xy
(-16, 9)*+{X=\{1,2\}} ="1"; 
(16, 9)*+{\{*\}=Y} ="2";
(0,-4)*+{\SSub}="3";
{\ar@{->}_{\tau} "1";"3"};
{\ar@{->}^{\varphi} "1";"2"};
{\ar@{->}^{\mu} "2";"3"};{\ar@{=>}_{\scriptstyle \Phi} (4,5)*{};(-0.4,3)*{}} 
\endxy  
\]
where $\varphi :X\to Y$ is the only possible map.
and set $\tau(1) =\tau(2) = \mu(*)= (p: \R\times \R\to \R) $, where
$p$ is the projection on the first factor.  
We define $\Phi=
\{\Phi_i:\mu(\varphi(i))\to \tau(i)\}_{i =1}^2$ as follows.  We choose
\[
\Phi_1:\mu(\varphi(1))= (p:\R\times \R \to \R) \to (p:\R\times \R \to
R) =\tau(1)
\]
 to
  be the map 
\[
\Phi_1 (x,u) = (x^2, u) 
\]
and  
\[
\Phi_2:\mu(\varphi(2))=(p:\R\times \R \to \R) \to (p:\R\times \R \to \R) =\tau(2)
\]
to be the map 
\[
\Phi_2 (x,u) = (x, u^2). 
\]
Then $\PP(\varphi, \Phi): (p:\R\times \R \to \R) \to (p:\R\times
\R \to R)^2$ is given by
\[
\PP(\varphi, \Phi) (x,u) = ((x^2, u), (x, u^2)).
\] 
We choose $s: M\to U$ to be the identity map.  This gives us an
interconnection map
\[
\nu: (\id:\R \to \R) \to (p:\R\times \R \to \R), \qquad \nu(x) = (x, x).
\]
We choose $\psi: (\id:\R \to \R)^2 \to (p:\R\times \R \to \R)^2$ to be 
\[
\psi(x_1, x_2) = ((x_1,x_2), (x_2, x_1)).
\]
We choose $f: M = \R \to \R^2 = M^2$ to be the map 
\[
f(x) = (x^2, x).
\]
It is easy to see that 
\[
\PP(\varphi, \Phi) \circ \nu = \psi\circ f.
\]
In this case Theorem~\ref{thm:main} tells us that given any three open
systems $G, F_1, F_2 \in \Crl (p:\R^2 \to \R)$ so that $G$ is
$\Phi_i$-related to $F_i$ ($i\in X$), the vector field $\nu^*G$ is
$f$-related to the vector field $\psi^* (F_1, F_2)$.  Consequently since
the image of $f$ is the parabola
\[
\{(x_1, x_2) \in \R^2\mid x_1 = (x_2)^2\},
\]
the parabola is an invariant submanifold of the vector field $\psi^*
(F_1, F_2)$.  Such an invariant submanifold can never arise in the
coupled cell networks formalism.

A reader may wonder the the vector space of triples $G, F_1, F_2 \in
\Crl (p:\R^2 \to \R)$ so that $G$ is $\Phi_i$-related to $F_i$ ($i\in
X$) is non-zero.  It is not hard to see that the space of such triples
is at least as big as the space $C^\infty (\R^2)$. Indeed, given a function
$g\in C^\infty (\R^2)$ let
\begin{gather*}
G(x,u) = \frac{1}{2} xg (x^2, u^2)\\
F_1 (v,u) = v g(v, u^2)\\
F_2 (x,w) =\frac{1}{2} xg (x^2, w).
\end{gather*}
Then 
\[
F_2 \circ \Phi_2 (x,u) = T(\Phi_2)_\st\, G(x,u) (=G(x,u))
\]
 and 
\[
(F_1 \circ \Phi_1) (x,u) =
x^2 g(x^2, u^2) = 2x G(x,u) = (T(\Phi_1)_\st \circ G) (x, u).
\]
Consequently the parabola $\{x_1= (x_2)^2\}$ is an invariant submanifold
for any vector field $v:\R^2 \to \R^2$ of the form
\[
v(x_1,x_2)  = (x_1 g(x_1, x_2^2), \frac{1}{2} x_2 g(x_2^2, x_1))
\]
for any function $g\in C^\infty (\R^2)$.\hfill $\Box$.
\end{example}

\begin{example} \label{ex:1.3} In Examples~\ref{ex:9.1} and
  \ref{ex:1.2} we started with two collections of open systems and
  ended up with two related closed systems.  It is easy to modify Example~\ref{ex:1.2}  so that the end result are two related open systems.   

We now carry out the modification. As before let $X =\{1,2\}$, $Y=
  \{*\}$, and let $\varphi: X\to Y$ be the only possible map.  Now
  consider the surjective submersion 
\[
q:\R^3 \to \R\qquad q(x,u,v) = x.
\]
Set $\mu(*) = \tau (1) = \tau (2) = (q:\R^3 \to \R)$.  Choose
$\Phi_i:\mu(*) \to \tau(i)$, $i=1,2$, to be the maps
\[
\Phi_1(x,u, v) = (x^2, u,v),\qquad \Phi_2(x,u, v) = (x, u^2,v).
\]
This defines a morphism
\[
\xy
(-10, 8)*+{\{1,2\}} ="1"; 
(10, 8)*+{\{*\}} ="2";
(0,-4)*+{\SSub}="3";
{\ar@{->}_{\tau} "1";"3"};
{\ar@{->}^{\varphi} "1";"2"};
{\ar@{->}^{\mu} "2";"3"};{\ar@{=>}_{\scriptstyle \Phi} (4,5)*{};(-0.4,3)*{}} 
\endxy 
\]
In the category of lists of submersions.
It is easy to see that 
\[
\PP(\varphi, \Phi) \,(x,u,v) = ((x^2, u,v), (x,u^2, v)).
\]
Choose the interconnection maps $\psi, \nu$ as follows:
\begin{gather*}
\nu:(p:\R^2 \to \R) \to (q:\R^3 \to \R),\qquad  \nu(x, v) = (x, x,v); \\
\psi:(p:\R^2 \to \R)^2 \to (q:\R^3 \to \R)^2,\qquad  
\psi ((x_1, v_1), (x_2,v_2)) = ((x_1, x_2 ,v), (x_2, x_1, v)).
\end{gather*}
Let $f: (p:\R^2 \to \R) \to (p:\R^2 \to \R)^2$ be the map
\[
f(x,v) = ((x^2, v),(x,v)).
\]
It is again easy to check that the equality 
\[
\PP(\varphi, \Phi) \circ \nu = \psi\circ f
\]
holds with our choices of $\psi,\nu, f, p$ and $\Phi$.  Therefore,
by Theorem~\ref{thm:main}, given three open systems $G,F_1, F_2\in
\Crl(q:\R^3 \to \R)$ so that $G$ is $\Phi_1$-related to $F_1$ and
$\Phi_2$-related to $F_2$, the open systems $\nu^*G$ is $f$-related
to the open system $\psi^*(F_1,F_2)$.
\end{example}

\section{Networks of manifolds} \label{sec:networks}

The first result of this section shows that the networks of manifolds
defined in \cite{DL1} (hence the coupled cell networks of Golubitsky,
Stewart et al.) are a special case of the networks of open systems
(Definition~\ref{def:network}).  In particular they are morphisms in
the colored operad $\cO((\SSubi)^\op$.  We then show that Theorem~3 of
\cite{DL1} (which is the main result of that paper) is a direct
consequence of Theorem~\ref{thm:main} above.  In particular we show
that fibrations of networks of manifolds
(Definition~\ref{def:fibration}) give rise to maps of networks of open
systems in the sense of Definition~\ref{def:map_of_networks}.  To give
credit where it is due, Definition~\ref{def:map_of_networks} and
Theorem~9.3 were directly inspired by Theorem~3 of \cite{DL1}.  We
start by setting up notation, which differs somewhat from the notation
of \cite{DL1}.

\begin{definition} A finite directed {\sf graph} $G$ is a pair of
  finite sets $G_1$ (arrows/edges), $G_0$ (nodes/vertices) and two
  maps $s,t: G_1 \to G_0$ (source and target). We write: $G =
  \{G_1\toto G_0\}$.
 \end{definition}

 \begin{definition}\label{def:net.mfld} A {\sf network of manifolds}
is a pair $(G,\cP)$ where $G = \{G_1\toto G_0\}$ is a finite graph and
$\cP: G_0\to \Man$ is a list of manifolds (i.e., $
G_0\xrightarrow{\cP}\Man$ is an object of the category $\FinSet/\Man$).

 A {\sf map of networks of manifolds} $\varphi:(G, \cP) \to
   (G' , \cP')$ is a map of graphs $\varphi: G\to G'$ so that
   $\cP'\circ \varphi = \cP$.
\end{definition}
Recall that since the category of manifolds $\Man$ has finite
products, we have a functor
\[
\PP= \PP_\Man:  (\FinSet/\Man)^{op} \to \Man,
\]
which assigns to a list $\mu: X\to \Man$ the corresponding product:
\[
\PP (X\xrightarrow{\mu}\Man)= \prod_{a\in X} \mu(a)
\]
(cf.\ Section~\ref{sec:lists}). In particular to every network
$(G, \cP)$ of manifolds in the sense of Definition~\ref{def:net.mfld}
the functor $\PP$ assigns the manifold $\PP(\cP)= \PP(G_0\xrightarrow{\cP}
\Man)$ which we think of as the {\sf total phase space} of the
network. And to every map of networks of manifolds $\varphi: (G,\cP))
\to (G',\cP'))$ the functor $\PP$ assigns a map
\[
\PP(\varphi):\PP(\cP')\to \PP (\cP).
\]
between their total phase spaces.

\begin{proposition}\label{prop:5.5}
A  network of manifolds $(G,\cP)$ encodes 
\begin{enumerate}
\item a list of submersions $I: G_0\to \SSub$ (i.e., an object of
  $\FinSet/\SSub$) and
\item an interconnection morphism $\cI_G:(\PP_\Man(\cP) \xrightarrow{\id}
  \PP_\Man(\cP)) \to \PP_\SSub(I) $.
\end{enumerate}
Consequently a network of manifolds in the sense of
Definition~\ref{def:net.mfld} does give rise to a network of open
systems in the sense of Definition~\ref{def:network}.
\end{proposition}

\begin{remark} Recall (Remark~\ref{rmrk:2.10}) that for any manifold
  $M$ the space $\Crl(M\xrightarrow{id}M)$ of control systems is the
  space $\scX(M)$ of vector fields on the manifold $M$.  Hence the
  interconnection map $\cI_G:(\PP_\Man(\cP) \xrightarrow{\id}
  \PP_\Man(\cP)) \to \PP_\SSub(I) $ gives rise to the linear map
\[
\cI_G^* :\Crl(\PP_\SSub(I)) \to \scX(\PP_\Man(\cP)) 
\]
from the space of the open systems on the product $\PP_\SSub(I)=\prod_{a\in
  G_0}I(a)$ of submersions to the space of vector fields on the manifold
$\PP_\Man(\cP)$.
 \end{remark}

\begin{remark}\label{rmrk:10.*}
  Recall that by Lemma~\ref{lem:5.14} a list of submersions $\tau:X\to
  \SSub$ gives rise to a canonical linear map $\Crl_\tau: \bigoplus
  _{x\in X}\Crl(\tau(x))\to \Crl(\prod_{x\in X} \tau(x)) $.
  Consequently for any network of manifolds $(G,\cP)$ we get a linear
  map
\[
\Crl_I: \bigoplus_{a\in G_0}\Crl(I(a)) \to \Crl(\PP_\SSub(I)),
\]
hence  the composite map 
\[
(\cI_G)^* \circ \Crl_I: \bigoplus_{a\in G_0}\Crl(I(a)) \to \scX(\PP_\Man(\cP).
\]
\end{remark}
\begin{proof}[Proof of \protect{Proposition~\ref{prop:5.5}}]
Given a node $a$ of a graph $G$ we associate two maps of finite sets:
\begin{itemize}
\item $\iota_a: \{a\} \hookrightarrow G_0$ and
\item $s|_{t\inv (a)}: t\inv(a) \to G_0$ (recall that $s, t:
  G_1\to G_0$ are the source and target maps of the graph $G$).
\end{itemize}
The set
$t\inv(a)$ is the collection of arrows of $G$ with target $a$ and $s$
sends this collection to the sources of the arrows. The composition
with $\cP: G_0 \to \Man$ gives us two lists of manifolds: 
\[
\cP\circ
\iota_a :\{a\}\to \Man
\] 
and 
\[
\xi_a:= \cP \circ s|_{t\inv (a)}: t\inv(a)
\to \Man.
\] 
Applying the functor $\PP$ gives us two manifolds: $\PP (\cP\circ
\iota_a)$, which is just $\cP(a)$, and $\PP (\xi_a)= \prod_{\gamma \in
  t\inv (a)} \cP(s(\gamma))$. We define the submersion $I(a)$ to be
the projection on the first factor %
\[
I(a):=(\cP(a)\times \PP (\xi_a) \xrightarrow{pr_1} \cP(a)).
\]
This gives us the desired list of submersions
\[
I:G_0\to\SSub, \qquad a\mapsto I(a).
\]
Since 
\[
I(a) = \left( \cP(a)\times \PP (t\inv (a)\xrightarrow{\xi_a} \Man)\to
\cP(a)\right),
\]
\[
\prod_{a\in G_0} I(a) =  \left( \prod_{a\in G_0}\cP(a)\times \prod_{a\in G_0} \PP (t\inv (a)\xrightarrow{\xi_a} \Man)\to
\prod_{a\in G_0}\cP(a)\right).
\]
By definition
\[
\PP_\SSub(I) = \prod_{a\in G_0} I(a).
\]
Therefore 
\[
\PP_\SSub(I) = \left((\PP(G_0\xrightarrow{\cP} \Man)\times
    (\PP(\bigsqcup{a\in G_0} \xrightarrow{\sqcup \xi_a} \Man)\to 
    (\PP(G_0\xrightarrow{\cP}\to \Man) \right).
\]
It follows that in order to construct an interconnection map 
\[
\cI_G:(\PP(\cP) \xrightarrow{\id}\PP(\cP)) \to\PP_\SSub(I) 
\]
 it suffices to construct a map 
\[
f: \PP_\Man(\cP) \to \prod_{a\in G_0} \PP_\Man(t\inv
(a)\xrightarrow{ \xi_a} \Man) .
\] 
This map $f$ too comes from a map of
finite sets. Namely, the family $\{s|_{t\inv (a)}:t\inv(a) \to
G_0\}_{a\in G_0}$ defines
\[
\sqcup s|_{t\inv (a)}: \bigsqcup _{a\in G_0}t\inv(a) \to G_0,
\]
and the diagram 
\[
\xy
(-15, 10)*+{ \bigsqcup t\inv(a) } ="1"; 
(15, 10)*+{G_0} ="2";
(0,-8)*+{\Man }="3";
{\ar@{->}_{\sqcup \xi_a} "1";"3"};
{\ar@{->}^{\sqcup (s|_{t\inv (a)})} "1";"2"};
{\ar@{->}^{\cP} "2";"3"};
\endxy 
\]
commutes. So set $f := \PP_\Man( \sqcup s|_{t\inv (a)})$.
\end{proof}

\begin{example} \label{ex:5.9}
Let $G$ be the graph 
\begin{center}
\tikzset{every loop/.style={min distance=15mm,in=-30,out=30,looseness=10}}
\begin{tikzpicture}[->,>=stealth',shorten >=1pt,auto,node distance=2cm,
  thick,main node/.style={circle,draw,font=\sffamily\bfseries}]
  \node[main node] at (2,-1) (1) {1};
  \node[main node] at (4,-1) (2) {2};
  \node[main node] at (6,-1) (3) {3};
  \path[]
  (1) edge [bend left] node {} (2)
  (2) edge [bend left] node [below] {} (1)
  (2) edge  node {}(3)  ;
\end{tikzpicture}.
\end{center}
Let $\cP$ be the function that assigns to every node the same manifold $M$.
Then 
\[
I(a) = (M\times M\xrightarrow{pr_1} M)
\]
for every $a\in G_0$, 
\[
\PP_\SSub( I) = (M\times M \to M)^3 \simeq ((M^2)^3 \to M^3)
\]
and the interconnection map $\cI_G: M^3 \to (M^2)^3
$ is given by
\[
\cI_G(x_1, x_2, x_3) = ((x_1,x_2),  (x_2, x_1), (x_3, x_2) )
\]
for all $(x_1, x_2, x_3) \in M^3$. Finally 
\[
(\cI_G)^* \circ
\Crl_{I}:\Crl (M\times M\to M)^{\oplus 3} \to \scX (M^3)
\]
(see Remark~\ref{rmrk:10.*}) is given by
\[
\left(\,(\cI_G)^* \circ \Crl_I)(w_1,w_2, w_3)\,\right)\, (x_1, x_2, x_3) =
\big(w_1(x_1, x_2), w_2(x_2,x_1), w_3(x_3,x_2)\big).
\]
\mbox{}\hfill $\Box$
\end{example}

In \cite{DL1, DL2} a class of maps of networks of manifolds was singled out.

\begin{definition}[Fibration of networks of manifolds]\label{def:fibration}
 A map $\varphi:G\to  G'$ of
  directed graphs is a {\sf graph fibration} if for any vertex $a$ of $ G$
  and any edge $e'$ of $ G'$ ending at $\varphi(a)$ (i.e., the target
  $t(e')$ of the edge $e'$ is 
  $\varphi(a)$) there is a
  unique edge $e$ of $ G$ ending at $a$ with $\varphi (e) = e'$.

  A map of networks of manifolds $\varphi:(G,\cP) \to
  (G',\cP')$ is a {\sf fibration} if $\varphi: G\to G'$ is a
  graph fibration.
\end{definition}

\begin{remark}
  A graph fibration $\varphi: G\to G'$ is in general neither injective
  nor surjective on vertices. However, for every vertex $a\in G_0$ it
  induces a bijection between the set $t\inv (a)$ of arrows of $G$
  with target $a$ and the set $t\inv (\varphi(a))$ of arrows of $G'$
  with target $\varphi(a)$.  
\end{remark}

The reason for singling out fibrations of networks of manifold is that
they give rise to maps of dynamical systems.  This is the main result
of \cite{DL1} and of \cite{DL2}. (In \cite{DL2} only the so called
``groupoid invariant'' vector fields were considered.  The requirement
of the groupoid symmetry turned out to be irrelevant and was dropped
in \cite{DL1}.)  As we have seen above, a network of manifolds is
a morphism in the operad $\cO((\SSubi)^\op)$ whose target happens to
be a fibration of the form $M\xrightarrow{id_M}M$.  With the benefit
of hindsight, the results of \cite{DL1} can be reformulated as the two
lemmas and the theorem below.

\begin{lemma}\label{lem:10.9}
Let $\varphi: (G,\cP) \to (G',\cP')$ be a fibration of networks of
manifolds.  Then for each vertex $a\in G_0$ we have isomorphism of
submersions
\[
\Phi(a): I(a) \to I'(\varphi(a)).
\]
Hence a fibration $\varphi$ gives rise to a morphism
\begin{equation}
\xy
(-10, 10)*+{G_0} ="1"; 
(10, 10)*+{G_0'} ="2";
(0,-2)*+{\SSub}="3";
{\ar@{->}_{I} "1";"3"};
{\ar@{->}^{\varphi} "1";"2"};
{\ar@{->}^{I'} "2";"3"};
{\ar@{=>}_{\scriptstyle \Phi} (4,6)*{};(-0.4,4)*{}} ; 
\endxy
\end{equation}
in the category $(\FinSet/\SSub)^\Leftarrow$.
\end{lemma}

\begin{proof}
Since $\varphi: G\to G'$ is a graph fibration, the restriction 
\[
\varphi|_{t\inv (a)}: t\inv (a)\to t\inv (\varphi(a))
\]
is a bijection for each vertex $a$ of the graph $G$. Since $\cP'\circ
\varphi = \cP$
\[
\phi(a):= \varphi|_{t\inv (a)}:(t\inv (a)\xrightarrow{\xi_a} \Man ) \to 
(t\inv (\varphi(a))\xrightarrow{ \xi_{\varphi(a)}} \Man)
\]
is an isomorphism in $\FinSet/\Man$. Here, as in the proof of
Proposition~\ref{prop:5.5}, 
\[
\xi_a = \cP\circ s|_{t\inv (a)}
\]
and $\xi_{\varphi(a)}$ is defined similarly.
Consequently
\[
\PP_\Man(\varphi|_{t\inv (a)}):\PP_\Man(\xi_a ) \to \PP_\Man (t\inv (\xi_{\varphi(a)})
\]
is an isomorphism in $\Man$. Thus for each $a\in G_0$ we have an
isomorphism of surjective submersions
\[
\Phi(a):\left( \PP_\Man(\xi_a ) \times \cP(a) \to
  \cP(a)\right)\quad
\xrightarrow{\quad} 
\quad \left( \PP_\Man( \xi_{\varphi(a)} )
  \times \cP'(\varphi(a))\to \cP'(\varphi(a))
\right)
\]
(note that $\cP(a) = \cP'(\varphi(a))$). The family of isomorphisms
$\{\Phi(a)\}_{a\in G_0}$ together with the map $\varphi$ define a
morphism \[
\xy
(-10, 10)*+{G_0} ="1"; 
(10, 10)*+{G_0'} ="2";
(0,-2)*+{\SSub }="3";
{\ar@{->}_{I} "1";"3"};
{\ar@{->}^{\varphi} "1";"2"};
{\ar@{->}^{I'} "2";"3"};
{\ar@{=>}_{\scriptstyle \Phi} (4,6)*{};(-0.4,4)*{}} ; 
\endxy 
\]
in $(\FinSet/\SSub)^\Leftarrow$.
\end{proof}

\begin{example} \label{ex:10.11}
Let $\varphi\colon G\to G'$ be the graph fibration
\begin{center}
\tikzset{every loop/.style={min distance=15mm,in=-30,out=30,looseness=10}}
\begin{tikzpicture}[->,>=stealth',shorten >=1pt,auto,node distance=2cm,
  thick,main node/.style={circle,draw,font=\sffamily\bfseries}]

  \node[main node] at (2,-1) (1) {1};
  \node[main node] at (4,-1) (2) {2};
  \node[main node] at (6,-1) (3) {3};
  \node at (1,-1) (rb){};
  \node at (-1,-1) (lb){};
  \node[main node] at (-3,-1) (circ){*};

  \path[]
  (1) edge [bend left] node {} (2)
  (2) edge [bend left] node [below] {} (1)
  (2) edge  node {}(3)
  (circ) edge [loop right=20] node {} (circ)
  (rb) edge node [above] {$\varphi$} (lb) ;
\end{tikzpicture}.
\end{center}
Define $\cP': G_0'\to \Man$ by setting $\cP'(*) = M$ for some manifold
$M$ on the single vertex $*$ of $G'$. Define $\cP: G_0\to \Man$ by
setting $\cP(i) = M$ for $i=1,2,3$. Then 
\[
\varphi:(G, \cP)\to (G', \cP')
\]
is a fibration of networks of manifolds.  As in Example~\ref{ex:5.9} we have a list of submersions $I= I_G: G_0 \to \SSub$ with 
\[
I(a) = (M\times M\xrightarrow{pr_1} M)
\]
for every $a\in G_0$.%

Similarly we have $I': G_0'= \{*\}\to \SSub$ given by 
\[
I'(*) = (M\times M\xrightarrow{pr_1} M).
\]
Tracing through the proof of Lemma~\ref{lem:10.9} we see that the maps 
\[
\Phi(a): I'(\varphi(a))\to I(a)
\]
of submersions are 
the identity maps for all $a\in G_0$.
Consequently the diagram
\[
\xy
(-10, 10)*+{G_0} ="1"; 
(10, 10)*+{G_0'} ="2";
(0,-2)*+{\SSub }="3";
{\ar@{->}_{I} "1";"3"};
{\ar@{->}^{\varphi} "1";"2"};
{\ar@{->}^{I'} "2";"3"};
{\ar@{=>}_{\scriptstyle \id} (4,6)*{};(-0.4,4)*{}} ; 
\endxy 
\]
commutes.
\end{example}

\begin{remark}\label{rmrk:graph}
  Note that for a fibration $\varphi: (G,\cP) \to (G',\cP')$ of
  networks of manifolds, the components $\Phi(a): I(a) \to
  I'(\varphi(a))$ of the natural transformation $\Phi:I'\circ \varphi
  \Rightarrow I$ are all isomorphisms of submersions.  Consequently
  the relations
\[
\Crl (\Phi(a)):\Crl(I'(\varphi(a))\tickar \Crl(I(a))
\]
are graphs of linear maps.  It follows from Lemma~\ref{lem:6.2} that
the relation
\[
\odot(\varphi, \Crl\circ \Phi):\bigoplus _{b\in G'_0} \Crl(I'(b) \tickar
\bigoplus _{a\in G_0}\Crl(I(a))
\]
is a graph of the linear map $\oplus(\varphi, \Crl\circ \Phi)$.
\end{remark}

\begin{remark}\label{rmrk:identify} 
We identify the category $\Man$ of manifolds with a subcategory
of the category $\SSub$ of surjective submersions.  Namely we identify
a manifold $M$ with the submersion $id_M:M\to M$ and a map $f:M\to N$
with the map of submersions $(f,f): (M\xrightarrow{id_M}M)\to
(N\xrightarrow{\id_N}N)$.  In particular given a morphism of lists of
manifolds
\[
\xy
(-10, 8)*+{X} ="1"; 
(10, 8)*+{Y} ="2";
(0,-4)*+{\Man}="3";
{\ar@{->}_{\tau} "1";"3"};
{\ar@{->}^{\varphi} "1";"2"};
{\ar@{->}^{\mu} "2";"3"};
\endxy
\]
we have a map of {\sf submersions}
\[
\PP(\varphi):\PP(\mu)\to \PP(\tau).
\]
\end{remark}

\begin{lemma}\label{lem:10.14}
 A fibration of networks of manifolds $\varphi: (G,\cP)
\to (G',\cP')$ gives rise to a map of networks of open systems
\begin{multline*}
((\varphi, id), \PP(\varphi) : (I':G'_0\to \SSub,
\cI_{G'}:\PP(\cP')\to \PP (I')) \to \\(I:G_0\to \SSub,
\cI_{G}:\PP(\cP)\to \PP( I)),
\end{multline*}
where, as in Remark~\ref{rmrk:identify} the category of manifolds is identified
with a subcategory of submersions.
\end{lemma}

\begin{proof}  By Lemma~\ref{lem:10.9} the fibration $\varphi$ gives rise to a map 
to a morphism
\[
\xy
(-10, 10)*+{G_0} ="1"; 
(10, 10)*+{G_0'} ="2";
(0,-2)*+{\SSub}="3";
{\ar@{->}_{I} "1";"3"};
{\ar@{->}^{\varphi} "1";"2"};
{\ar@{->}^{I'} "2";"3"};
{\ar@{=>}_{\scriptstyle \Phi} (4,6)*{};(-0.4,4)*{}} ; 
\endxy
\]
in the category $(\FinSet/\SSub)^\Leftarrow$.  Consequently we have a
map of submersions
\[
\PP(\varphi, \Phi): \PP(I') = \prod_{b\in G_0'} I'(b)\to \prod_{a\in G_0} I(a) =\PP(I).
\]
We need to check that the diagram
\[
\xy
(-18,8)*+{ \prod_{b\in G'_0} I'(b)}="1";
(18,8)*+{ \prod_{a\in G_0} I(a))}="2";
(-18,-10)*+{ \PP(\cP')}="3";
(18,-10)*+{\PP(\cP)}="4";
{\ar@{->}^{\PP(\varphi, \Phi)} "1";"2"};
{\ar@{->}^{\cI'} "3";"1"};
{\ar@{->}_{\cI} "4";"2"};
{\ar@{->}_{\PP(\varphi)} "3";"4"};
\endxy.
\]
commutes in the category $\SSub$ of submersions, hence defines a
2-cell in the double category $\SSub^\Box$.

Since $\varphi$ is a fibration of networks of manifolds, for
each node $a$ of the graph $G$ we have a commuting square in
$\FinSet/\Man$:
\begin{equation}\label{eq:square2}
\xy
(-10,8)*++{t\inv(a)}="1";
(30,8)*++{G_0}="2";
(-10,-8)*++{t\inv (\varphi(a))}="3";
(30,-8)*++{G'_0}="4";
(14,-24)*++{\Man}="5";
{\ar@{->}^{s|_{t\inv(a)}} "1";"2"};
{\ar@{->}_{\phi(a)} "1";"3"};
{\ar@{->}^{\varphi} "2";"4"};
{\ar@{->}^{s|_{t\inv(\varphi(a))}} "3";"4"};
{\ar@{->}^<<<<<{\xi_a  = \cP\circ s|_{t\inv (a)}
\quad} "1";"5"};
{\ar@{->}_<<<<<{\cP} "2";"5"};
{\ar@{->}_{\xi_{\varphi(a)} } "3";"5"};
{\ar@{->}^{\cP'} "4";"5"};
\endxy .
\end{equation}
Here as before $\xi_a = \cP\circ s|_{t\inv (a)}$ and $\xi_{\varphi(a)}
= \cP' \circ s|_{t\inv (\varphi(a))}$.  We now drop the maps to $\Man$
to reduce the clutter and only keep track of maps of finite sets. By
the universal property of coproducts the diagrams \ref{eq:square2}
define a unique map
\[
\Psi: \bigsqcup_{a\in G_0} t\inv (a) \to \bigsqcup_{b\in G'_0} t\inv (b)
\]
in $\FinSet/\Man$ 
so that the diagram
\[
\xy
(-20, 10)*+{t\inv(a)\quad}="1"; 
(20,10)*+{ \bigsqcup_{x\in G_0} t\inv (x)}="2";
(65,10)*+{G_0}="22";
(-20, -10)*+{t\inv(\varphi(a))\quad}="3"; 
(20,-10)*+{\bigsqcup_{b\in G'_0} t\inv (b)}="4";
(65,-10)*+{G_0'}="24";
{\ar@{->}^{ }"1";"2"};
{\ar@{->}_{\phi(a)} "1";"3"};
{\ar@{->}_{} "3";"4"};
{\ar@{->}^{\Psi} "2";"4"};
{\ar@{->}^{\varphi} "22";"24"};
{\ar@{->}_{\sqcup_{b\in G'_0} s|_{t\inv(b)}} "4";"24"};
{\ar@{->}^{\sqcup_{a'\in G_0} s|_{t\inv(a')}} "2";"22"};
\endxy
\]
commutes for all nodes $a\in G_0$. Applying the functor $\PP= \PP_\Man$ gives
the commuting diagram in $\Man$:
\[
\xy
(15,10)*+{ \PP( \bigsqcup_{a\in G_0} t\inv (a))}="2";
(65,10)*+{\PP(G_0)}="22";
(15,-10)*+{ \PP(\bigsqcup_{b\in G'_0} t\inv (b))}="4";
(65,-10)*+{\PP(G_0')}="24";
{\ar@{<-}_{\PP(\Psi)} "2";"4"};
{\ar@{<-}^{\PP (\varphi)} "22";"24"};
{\ar@{<-}_{\qquad\PP(\sqcup_{b\in G'_0} s|_{t\inv(b)})} "4";"24"};
{\ar@{<-}^{\qquad\PP(\sqcup_{a'\in G_0} s|_{t\inv(a')})} "2";"22"};
\endxy.
\]
The fact that $\PP$ takes coproducts to products and the universal
properties ensure that
\[
(\PP(\Psi)\times \PP (\varphi), \PP (\varphi)) = \PP_\SSub(\varphi, \Phi).
\]
\end{proof}

The main result of \cite{DL1} %
can now be stated as follows.
\begin{theorem}[\protect{\cite[Theorem~3]{DL1}}]\label{thm3ofDL1}
  A fibrations of networks of manifolds $\varphi:(G, \cP) \to
  (G',\cP')$ gives rise to a 2-cell
\begin{equation}\label{eq:10.17}
\begin{tikzcd}[sep=large]
\bigoplus_{b\in G_0'} \Control(I'(b))\ar[r,tick,"{\odot(\varphi, \Crl\circ \Phi)}"]
	\ar[d,"{(\cI_{G'})^*\circ \Crl_{I'}}"']
&
	\bigoplus _{a\in G_0}\Control(I(a))\ar[d,"{ (\cI_G)^*\circ \Crl_{I}}"]
\\
	\scX(\PP_\Man(\cP'))\ar[r,tick,"\scX(\PP_\Man(\varphi))"']
&
	\scX(\PP_\Man(\cP))\ar[ul,phantom,"\Box"]	
\end{tikzcd}
\end{equation}
in $\RelVect$.  Here as before $\Phi: I'\circ \varphi\Rightarrow I$ is
the natural transformation of Lemma~\ref{lem:10.9}.

Consequently 
for every family open systems
$(w_b)_{b\in G_0'}\in \bigoplus_{b\in G_0'} \Control(I'(b))$ we have a
map of dynamical systems
\begin{multline}
\PP(\varphi): \left((\PP_\Man(\cP'), ((\cI_G')^*\circ
  \Crl_{I'})\,((w_b))\right) \xrightarrow{\quad} \\
\left(
(\PP_\Man(\cP),
((\cI_G)^*\circ \Crl_{I}\circ \oplus(\varphi, \Crl\circ \Phi))\,((w_b))
\right)
\end{multline}
(q.v.\ Remark~\ref{rmrk:graph}).
\end{theorem}

\begin{proof}
By Lemma~\ref{lem:10.14} a fibration of networks of manifolds $\varphi: (G,\cP)
\to (G',\cP')$ gives rise to a map of networks of open systems
\begin{multline*}
((\varphi, id), \PP(\varphi) : (I':G'_0\to \SSub,
\cI_{G'}:\PP(\cP')\to \PP (I')) \to \\(I:G_0\to \SSub,
\cI_{G}:\PP(\cP)\to \PP( I)),
\end{multline*}  
By Theorem~\ref{thm:main} the map of networks of open
systems give rise to the 2-cell \eqref{eq:10.17} in the double
category $\RelVect$.

By Remark~\ref{rmrk:graph} the relation $\odot(\varphi, \Crl\circ
\Phi)$ is the graph of the linear map $\oplus (\varphi, \Crl\circ
\Phi)$.  The fact that \eqref{eq:10.17} is 2-cell in $\RelVect$
translates into the conditions that for any tuple $(w_b)_{b\in
  G_0'}\in \bigoplus_{b\in G_0'} \Control(I'(b))$ the vector fields
$(\cI_{G'}^*\circ \Crl_{I'})\,((w_b))$ and $(\cI_G)^*\circ
\Crl_{I}\circ \oplus(\varphi, \Crl\circ \Phi))\,((w_b))$ are
$\PP(\varphi)$-related.
\end{proof}
\begin{example}
We illustrate Theorem~\ref{thm3ofDL1} by considering  the fibration $\varphi:(G, \cP)\to (G', \cP')$ of networks of manifold of 
 Example~\ref{ex:10.11}:
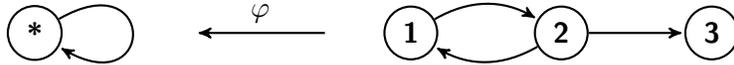
\begin{center}
\tikzset{every loop/.style={min distance=15mm,in=-30,out=30,looseness=10}}
\begin{tikzpicture}[->,>=stealth',shorten >=1pt,auto,node distance=2cm,
  thick,main node/.style={circle,draw,font=\sffamily\bfseries}]

  \node[main node] at (2,-1) (1) {1};
  \node[main node] at (4,-1) (2) {2};
  \node[main node] at (6,-1) (3) {3};
  \node at (1,-1) (rb){};
  \node at (-1,-1) (lb){};
  \node[main node] at (-3,-1) (circ){*};

  \path[]
  (1) edge [bend left] node {} (2)
  (2) edge [bend left] node [below] {} (1)
  (2) edge  node {}(3)
  (circ) edge [loop right=20] node {} (circ)
  (rb) edge node [above] {$\varphi$} (lb) ;
\end{tikzpicture}.
\end{center}
with $\cP'(*) = M$ for some manifold $M$ and with $\cP(a) = M$ for all
$a\in G_0=\{1,2,3\}$.  As we have seen in Example~\ref{ex:10.11} the list
$I:G_0\to \SSub$ is given by
\[
I(a) = (M\times M\xrightarrow{pr_1} M)
\]
for every $a\in G_0$ and $I': G_0'= \{*\}\to \SSub$ given by 
\[
I'(*) = (M\times M\xrightarrow{pr_1} M).
\]
Consequently 
\[
\PP_\SSub( I) = (M\times M \to M)^3 \simeq
((M \times M)^3 \to M^3).
\]
The interconnection map $\cI_G:  M^3 \to (M^2)^3 $ is given
by
\[
\cI_G (x_1, x_2, x_3) = ((x_1,x_2),  (x_2, x_1), (x_3, x_2) )
\]
for all $(x_1, x_2, x_3) \in M^3$.
The map  
\[
(\cI_G)^* \circ \Crl_{I}:\Crl (M\times M\to
M)^{\oplus 3} \to \scX (M^3)
\]
 is given by
\[
\left((\cI_G)^* \circ \Crl_{(G_0, I)})(w_1,w_2, w_3)\right)\, (x_1, x_2, x_3) =
\big(w_1(x_1, x_2), w_2(x_2,x_1), w_3(x_3,x_2)\big).
\]
Finally $(\cI_G)^* \circ \Crl_I: \Crl (M\times M\to
 M)^{\oplus 3} \to \scX (M^3)$ is given by
\[
\left((\cI_G)^* \circ \Crl_I)(w_1,w_2, w_3)\right)\, (x_1, x_2, x_3) =
\big(w_1(x_1, x_2), w_2(x_2,x_1), w_3(x_3,x_2)\big)
\]
for all $(x_1,x_2, x_3)\in M^3$ and all $(w_1,w_2, w_3)\in \Crl
(M\times M\to M)^{\oplus 3}$.  

On the other hand
\[
\PP_\SSub(I') = (M\times M \to M),
\]
and the interconnection map  $\cI_{G'}:  M \to M^2 $ is given by
\[
\cI_{G'} (x) = (x,x)
\]
for all $x\in M$.  Consequently 
\[
(\cI_{G'})^* \circ \Crl_{I'}: \Crl (M\times M\to M) \to \scX (M)
\] 
is given by 
 \[
\left((\cI_{G'})^* \circ \Crl_{I'}(w)\right)\, (x) =w(x, x).
 \]
 for all $w\in \Crl(M\times M\to M) = \oplus _{b\in G_0'}\Crl(I'(b))$
 and all $x\in M$.

  Observe first that 
\[
\PP_\SSub(\varphi, \id): (M\times M \to M) \to (M\times M \to M)^3 
\]
is the diagonal map.  Next note that the linear map
\[
\oplus(\varphi, \Crl\circ \id  ):
\Crl (M\times M\to M) \to \Crl (M\times M\to M) ^{\oplus 3}
\]
(the graph of which is the relation $\odot(\varphi, \Crl\circ \id  )$)
is given by 
\[
\oplus(\varphi, \Crl\circ \id  ) (w ) = (w, w, w).
\]
By Theorem~\ref{thm3ofDL1} for any open
system $w\in\Crl (M\times M\to M)$ the vector fields
\[
v:M\to TM, \qquad v(x) := w(x,x)\quad (=\left((\cI_{G'})^* \circ \Crl_{I'}(w)\right)\, (x))
\]
and 
\begin{eqnarray*}
u: M^3 \to TM^3, \qquad u(x_1, x_2, x_3) 
&:=& ((w(x_1, x_2), w(x_2,x_1), w(x_3,x_2))\\
\Big(&=&\left(\right((\cI_{G})^* \circ \Crl_{I} \circ 
\oplus(\varphi, \Crl\circ id)\left)\,(w)\right)\, (x_1, x_2, x_3))\,\Big)
\end{eqnarray*}
are $\PP_\Man (\varphi)$-related.  The fact that $v$ is $\PP_\Man
(\varphi)$-related to $u$ for any choice of an open system $w\in
\Crl(M\times M\to M)$ can also be checked directly.
\end{example}

\end{document}